\documentclass[12pt]{amsart}
\makeatletter
\def\@settitle{\begin{center}
\large
\baselineskip14\p@\relax
\bfseries
\@title
\end{center}
}
\makeatother

\makeatletter
\providecommand*{\approxident}{\mathrel{\mathpalette\@approxident\sim}} 
\newcommand*{\@approxident}[2]{\sbox0{$#1\vcenter{}$}
\sbox2{$\m@th#1\equiv$}\dimen2=\dimexpr\ht2 - \ht0\relax
\sbox4{$\m@th#1\sim$}\dimen4=\dimexpr\ht4-\ht0\relax
\dimen0=\dimexpr-\ht4-\dp4+\dimen2\relax
\vcenter{\offinterlineskip \copy4 \kern\dimen0\copy4\kern\dimen0\copy4\ifdim\dp4=\z@
\kern\dimexpr -\ht0+\dimen4\relax\fi}} 
\makeatother

\usepackage{amsmath}
\usepackage{float}
\usepackage{amsthm}
\usepackage{tikz}
\usepackage{hyperref}

\usepackage{tabularx,multirow,makecell}
\newcolumntype{Y}{>{\centering\arraybackslash}X}
\usepackage{setspace}
\usepackage{ltablex}

\newcounter{NN}\numberwithin{NN}{section}

\makeatletter
\@addtoreset{equation}{subsection}
\makeatother

\renewcommand\labelenumi{\rm (\roman{enumi})}
\renewcommand\theenumi{\rm (\roman{enumi})}

\usepackage{mathrsfs, amssymb, array,amsfonts, amsmath,upgreek,color}
\usepackage[matrix,arrow,curve]{xy}
\usepackage{hyperref}

\swapnumbers
\theoremstyle{plain}
\newtheorem{theorem}[subsection]{Theorem}
\newtheorem{stheorem}[equation]{Theorem}
\newtheorem{lemma}[subsection]{Lemma}

\newtheorem{sproposition}[equation]{Proposition}
\newtheorem{sproposition-definition}[equation]{Proposition-Definition}
\newtheorem{proposition}[subsection]{Proposition}

\newtheorem{scorollary}[equation]{Corollary}
\newtheorem{slemma}[equation]{Lemma}

\theoremstyle{definition}

\newtheorem{sremark}[equation]{Remark}

\newtheorem{sexample}[equation]{Example}

\newtheorem{pusto}[subsection]{}

\newcommand{\NE}{\overline{\operatorname{NE}}}
\newcommand{\EFF}{{\operatorname{Eff}}}

\newcommand{\p}{\mathrm{p}_{\mathrm{a}}}
\newcommand{\dd}{\mathrm{d}}

\newcommand{\Sing}{\operatorname{Sing}}

\newcommand{\Pic}{\operatorname{Pic}}
\newcommand{\Cl}{\operatorname{Cl}}
\newcommand{\Bs}{\operatorname{Bs}}

\newcommand{\rk}{\operatorname{rk}}

\newcommand{\PP}{\mathbb{P}}
\newcommand{\QQ}{\mathbb{Q}}
\newcommand{\ZZ}{\mathbb{Z}}
\newcommand{\CC}{\mathbb{C}}

\newcommand{\FF}{\mathbb{F}}
\newcommand{\NNN}{{\mathscr{N}}}
\newcommand{\OOO}{{\mathscr{O}}}
\newcommand{\EEE}{{\mathscr{E}}}
\newcommand{\LLL}{{\mathscr{L}}}
\newcommand{\HHH}{{\mathscr{H}}}
\newcommand{\type}[1]{$\mathrm{#1}$}

\newcommand{\xref}[1]{\textup{\ref{#1}}}

\newcommand{\g}{{\mathrm{g}}}

\begin{document}
\author{Yuri Prokhorov}
\thanks{The study has been funded within the framework of the HSE University Basic Research Program and the Russian Academic Excellence Project ``5-100''.}
\address{Steklov Mathematical Institute of Russian Academy of Sciences, Moscow, Russia
\newline
Faculty of Mathematics, Moscow State Lomonosov University, Russia
\newline
Laboratory of Algebraic Geometry, National Research University Higher School of Economics, Moscow, Russia
}
\email{prokhoro@mi-ras.ru}
\title{Rationality of Fano threefolds\\ with terminal Gorenstein singularities, I}

\begin{abstract}
We classify some special classes of non-rational Fano threefolds with terminal singularities. In particular, all such hyperelliptic and trigonal varieties are found.
\end{abstract}
\maketitle

\section {Introduction}
The question of rationality of varieties with ample anticanonical class is a classic problem going back to the works of G. Fano.
Currently, this problem is basically solved for three-dimensional \emph{nonsingular} Fano varieties, see
\cite{Beauville1977}, \cite{Clemens-Griffiths}, \cite{Iskovskih-Manin-1971b}, \cite{IP99}.
On the other hand, from the modern minimal model program point of view one has to consider Fano varieties with singularities.
In this paper we discuss the rationality of Fano threefolds with \emph{terminal Gorenstein} singularities.
We concentrate on the case of varieties with Picard number 1. In this situation, the most studied case is the the case of quartics in $\PP^4$
\cite{Pukhlikov-1989-e}, \cite{Corti2004a}, \cite{Kaloghiros2012}. Double spaces branched in a quartic and sextic are also relatively well studied from this point of view
(see \cite{Voisin1988}, \cite{Prz-Ch-Shr:DS}, \cite{Cheltsov2010a}, and the references therein).
The rationality of Fano threefolds with \emph{canonical} Gorenstein singularities was studied in \cite{Prokhorov-2004c}.

The main result of this paper is the following.
\begin{theorem}\label{thm:main}
Let $X$ be a Fano threefold with \textup(at worst\textup) terminal Gorenstein singularities and $\Pic(X)\simeq\ZZ$.
Suppose that $X$ is not rational.

If the group $\Pic(X)$ is not generated by the canonical class $K_X$ \textup(that is, $\iota(X)\ge2$\textup), then
one of the following possibilities occurs:
\begin{enumerate}
\renewcommand\labelenumi {\rm (\alph {enumi})}
\renewcommand\theenumi {\rm (\alph {enumi})}
\item
$X=X_3\subset\PP^4$ is a smooth cubic hypersurface, any variety of this type is not rational \cite{Clemens-Griffiths};
\item
$X=X_2\to\PP^3$ is a double space of index $2$, a general \textup(in particular, any smooth\textup) variety of this type is not rational
\cite{Beauville1977}, \cite{Tjurin1979}, \cite{Voisin1988}, \cite{Prz-Ch-Shr:DS};

\item
$X=X_1\to W$ is a Veronese double cone, a general \textup(in particular, any smooth\textup) variety of this type is not rational
\cite{Tjurin1979}, \cite{Grinenko2004}.
\end{enumerate}
Let $\Pic(X)=\ZZ\cdot K_X$ and let $g=\g(X)$ be the genus of $X$ \textup(see \eqref{eq:def-g}\textup).
Suppose that $g=\g(X)\ge 5$. Then the following assertions hold.
\begin{enumerate}
\item
The anticanonical divisor $-K_X$ is very ample and defines an embedding $X=X_{2g-2}\subset\PP^{g+1}$.
\item
If the variety is $X$
trigonal \textup(i.e. its anticanonical model $X=X_{2g-2}\subset\PP^{g+1}$ is not an intersection of quadrics\textup), then $X$ belongs to one of the two families described in Examples~\xref{example:trigonal:g=5} and~\xref{example:trigonal:g=6}.
A general variety of these families is not rational \textup(and the singular locus of $X$ consists of one ordinary double point\textup).
\item
If $X$ is non-trigonal and the anticanonical model $X=X_{2g-2}\subset\PP^{g+1}$ contains a plane, then $\g(X)=6$
and the variety $X$ is birationally equivalent to a smooth three-dimensional cubic $V_3\subset\PP^4$ under the map 
$V_3\dashrightarrow X_{2g-2}\subset\PP^{g+1}$ geven by the linear system of quadrics passing through a connected reduced curve $\Gamma\subset V_3\subset\PP^4$ of degree $3$
and arithmetic genus $0$ \textup(see Theorem~\xref{th:plane}\textup). Any variety of this type is not rational.
\item
If $X$ is not trigonal and the anticanonical model $X=X_{2g-2}\subset\PP^{g+1}$ does not contain planes,
then one can apply to $X$ an inductive birational construction \eqref{eq:constr} which increases the genus.
\item
If $X$ is locally $\QQ$-factorial, then either $\g(X)\le 6$ or $\g(X)=8 $ and $X$ is smooth.
\end{enumerate}
\end{theorem}

Recall that the genus of Fano threefolds satisfying the conditions $\uprho(X)=\iota(X)=1$ does not exceed 12 and is not equal to 11.
Therefore, \eqref{eq:constr} can give very explicit results.
A detailed analysis of the construction \eqref{eq:constr} will be given in the second part of the paper.

The author would like to thank the referee for careful reading the manuscript and constructive criticism. 

\section {Preliminaries} 
\subsection{Set-up.} 
\begin{itemize} 
\item 
By $\FF [d_1,\dots, d_n]$ we denote a rational scroll, the projectivization on $\PP^1$ of the vector bundle $\oplus\OOO_{\PP^1}(d_i)$. Usually, $M$ denotes the tautological divisor and $F$ denotes the class of a fiber of the projection to $\PP^1$. 
\item 
For a subvariety $Z\subset\PP^N $, $\langle Z\rangle$ denotes its linear span. 
\item 
By $\uprho(X)$ we denote the rank of the Picard group. 
\item 
For a birational map $f:V\dashrightarrow W$ and a linear system $\LLL$ of Weil divisors on $W$ its (birational) transform on $V$ is usually denoted by $\LLL_V$.
\end{itemize}
\subsection{Singularities}
Standard notions and facts from the   minimal model theory are used (see, e.g., \cite{Kollar-Mori-1988}).
\begin{stheorem}[{\cite[Th.~1.1]{Reid1983}}]
A three-dimensional Gorenstein terminal singularity $X\ni P$ is isolated of type \type{cDV}, i.e. its general hyperplane section $H\ni P$ is a
Du Val singularity. In particular, $X\ni P$ is analytically isomorphic to a hypersurface singularity of multiplicity $2$.
\end{stheorem}

\begin{scorollary}[{\cite[Lemma 5.1]{Kawamata-1988-crep}}]
\label{cor:Weil-Cartier}
Let $X$ be a three-dimensional variety with terminal Gorenstein singularities and let $D$ be an integral Weil divisor on $X$ such that $nD$ is a Cartier divisor for some integer $n$.
Then $D$ is a Cartier divisor.
\end{scorollary}
\subsection{Fano varieties}
Let $X$ be a Fano variety. By $\iota(X)$ we denote its \emph{index}, that is, the largest integer such that for some Cartier divisor $H_X$ the equality $-K_X=\iota H_X$ holds.
Further, the number $\dd(X):=H_X^{\dim X}$ is called the \emph{degree} of a Fano variety $X$.
\begin{sproposition-definition}[{\cite[Cor.~2.1.14]{IP99}}]
\label{prop:g}
Let $X$ be a Fano threefold with canonical Gorenstein singularities.
Denote
\begin{equation}
\label{eq:def-g}
\g(X):=\frac12 (-K_X)^3+1.
\end{equation}
Then $\g(X)$ is an integer, $\g(X)\ge2$, and
\[
\dim|-K_X|=\g(X)+1.
\]
\end{sproposition-definition}
The number $\g(X)$ is called the \emph{genus} of $X$.

\emph{A weak Fano variety} is a (projective) variety whose anticanonical divisor is nef and big.
Three-dimensional Fano varieties of index $2$ are called \emph{del Pezzo threefolds}.

\subsection{Extremal contractions}
\begin{stheorem}[\cite{Mori-1982}, \cite{Cutkosky-1988}]
\label{classification:extremal-rays}
Let $X$ be a three-dimensional projective variety with terminal factorial singularities and let $R$ be an extremal
ray on $X$.
Let $\ell $ be the minimal rational curve on $X$ generating $R$ and let
$f:X\to Y$ be the contraction of the ray $R$.

\par\medskip\noindent
If $R$ is not nef, then $f$ is birational and contracts 
an irreducible exceptional divisor $E$ so that $f$ is the blowup of 
$f(E)$ \textup(with reduced structure\textup). One of the following holds:
\par\medskip\noindent
\begin{longtable} {c|p {0.75\textwidth}|c}
{\rm type} &\multicolumn {1} {c|} {description} & $-K_X\cdot l $
\\\hline
\type{B_1} &
$f(E)$ is a curve with planar singularities, $\ell$ is a fiber of a ruled surface $E$,
$Y$ is smooth along $f(E)$ & $1$
\\\hline
\type{B_2} &
$f(E)$ is a point, $X$ is smooth along $E$, $Y$ is smooth at $f(E)$, $\ell $ is a line on $E\simeq\PP^2$ & $2$
\\\hline
\type{B_3} &
$f(E)$ is an ordinary double point, $X$ is smooth along $E$, $E\simeq\PP^1\times\PP^1$,
$\OOO_E (E)=\OOO_{\PP^1\times\PP^1}(-1,-1)$,
$\ell $ is a generator of
$\PP^1\times\PP^1$ & $1$
\\\hline
\type{B_4} &
$f(E)$ is a double point, $E$ is a quadratic cone in $\PP^3$,
$\OOO_E (E)=\OOO_{E}(-1)$, $\ell $ is a generator of the cone & $1$
\\\hline
\type{B_5} &
$f(E)$ is a non-Gorenstein point of multiplicity $4$, $X$ is smooth along $E$,
$\OOO_E (E)=\OOO_{\PP^2}(-2)$,
$\ell $ is a line on $E\simeq\PP^2$ & $1$
\end{longtable}

\par\medskip\noindent
If $R$ is nef, then $f:X\to Y$ is a
Fano-Mori fibration, $Y$ is smooth, and one of the following holds:
\par\medskip\noindent
\begin{longtable} {c|c|p {0.65\textwidth}|c}
{\rm type} & $\dim Y$ & $f:X\to Y$ & $-K_X\cdot l $
\\\hline
\type{C_1} &\multirow2 {*} {$2$} & conic bundle
with a nontrivial discriminant curve & $1$
\\\cline {1-1}\cline {3-4}
\type{C_2} & & $\PP^1$-bundle & $2$
\\\hline
\type{D_1} &\multirow3 {*} {$1$} & del Pezzo fibration of degree $\le 6$ & $1$
\\\cline {1-1}\cline {3-4}
\type{D_2} && fibration into quadrics & $2$
\\\cline {1-1}\cline {3-4}
\type{D_3} && $\PP^2$-bundle & $3$
\\\hline
\type{F} & $0$ & $X$ is a Fano threefold with $\Pic(X)\simeq\ZZ$ & $\iota(X)$
\end{longtable}
\end{stheorem}

\begin{scorollary}\label{cor:rat}
Let $X$ be a three-dimensional projective rationally connected variety with terminal factorial singularities and let $f:X\to Y$ be an 
extremal Mori contraction. If $f$ is of type \type{C_2}, \type{D_2} or \type{D_3}, then $X$ is rational.
If $f$ is of the type \type{D_1} and \textup(anticanonical\textup) degree of the generic fiber $\ge 5$, then $X$ is rational.
\end{scorollary}

\begin{proof}
In the case of $\dim Y=1$, the question reduces to the question of the rationality of the generic scheme fiber (which is a del Pezzo surface over the functional field).
This problem is studied well (see, e.g., \cite[Ch.~IV]{Manin-Cubic-forms-e-I}).
\end{proof}

\begin{scorollary}\label{cor:cb-fano}
Let $X$ be a weak Fano threefold with terminal factorial singularities such that the map to its anticanonical model contracts at most a finite number of curves and let $\pi:X\to Y$ be an extremal
contraction to a surface. Then $Y$ is a smooth del Pezzo surface.
\end{scorollary}

\begin{proof}[Outline of the proof]
By Theorem~\ref{classification:extremal-rays} the surface $Y$ is smooth.
Next, we follow the proof of Proposition 5.2 (i) of \cite{Prokhorov-2005a}. I is well-known that the following numerical equivalence holds
\[
\pi_* K_{X}^2\equiv -4K_Y-\Delta,
\]
where $\Delta$ is the discriminant curve (see, e.g., \cite[Lemma~3.11]{P:rat-cb:e}). We put $A:=\pi_* K_{X}^2$. Thus,
\[
-4K_Y\equiv A+\Delta.
\]
For any curve $C\subset Y$ one has $A\cdot C>0$. Indeed, the projection formula yields $A\cdot C=K_X^2\cdot \pi^*C$, where 
the last expression is strictly positive because the linear system $|-nK_{X}|$ is base point free and contracts only a finite number of curves.
Since $A$ is numerically equivalent to an effective divisor, it is ample.
To prove the ampleness of $-K_Y$, it suffices to show that $K_Y\cdot C <0$ for any curve $C\subset Y$.
Suppose that $K_Y\cdot C\ge 0$.
Then $\Delta\cdot C <0$ and $(K_Y+\Delta)\cdot C <0$. In particular, this means that $C$ is a component of $\Delta$, $\p (C)=0$, and
$\Delta'\cdot C <2$, where $\Delta':=\Delta-C$. 
Thus $C$ is a smooth rational curve which either is a connected component of the discriminant $\Delta$ or meets the complement 
$\Delta'$ transversally at one point. But then the Hilbert scheme parametrizing components of fibers over points $c\in C$ splits.
This contradicts the condition $\uprho(X/Y)=1$
(see, e.g., \cite[3.8-3.9]{P:rat-cb:e}).
\end{proof}

\subsection{Anticanonical divisor}

\begin{stheorem}[{\cite{Shokurov1979}}, {\cite{Reid-Kaw-1983}}, \cite{Mella-1999}]
\label{thm:ge}
Let $X$ be a Fano threefold with canonical Gorenstein singularities.
Then a general divisor $S\in|-K_X|$ is a normal surface with at worst Du Val singularities.
\end{stheorem}

It is clear that a general divisor $S\in|-K_X|$ is a (possibly singular) K3 surface.
From the classical results \cite{Saint-Donat-1974} on linear systems on K3 surfaces one obtains the following.

\begin{stheorem}[{\cite{Shin1989}}]\label{thm:bs}
Let $X$ be a Fano threefold with canonical Gorenstein singularities.
Suppose that the base locus $\Bs|-K_X|$ of the anticanonical linear system $|-K_X|$
is non-empty and let $S\in|-K_X|$ be a general divisor. Then $\iota(X)=1$ and one of the following cases occurs:
\begin{enumerate}
\item\label{thm:bs1}
$\Bs|-K_X|$ \textup(as a scheme\textup) is a smooth rational curve contained in the smooth locus of $X$ and a general element $S\in|-K_X|$ is smooth along $\Bs|-K_X|$;
\item\label{thm:bs0}
$\Bs|-K_X|$ consists of a single point $P$ which is singular for both $X$ and $S$. Moreover, in this case, $P\in S$ is a singularity of type \type{A_1},
and $-K_X^3=2$.
\end{enumerate}
\end{stheorem}

\begin{scorollary}\label{cor:bs}
Under the conditions of Theorem~\xref{thm:bs}, assume that the singularities of $X$ are terminal.
Then a general divisor $S\in|-K_X|$ is smooth, except for the case~\xref{thm:bs}\xref{thm:bs0}.
\end{scorollary}

\section {MMP on three-dimensional weak Fano threefolds}
This section is devoted to the proof of the following theorem.
\begin{theorem}\label{theorem:constr}
Let $X$ be a Fano threefold with terminal Gorenstein singularities.\footnote {Here we do not assume that $\uprho(X)=1$.}
Suppose that $X$ does not contain surfaces of anticanonical degree $1$, i.e. surfaces $S$ such that $(-K_X)^2\cdot S=1$.
Then there is a diagram
\begin{equation}\label{eq:constr}
\vcenter {
\xymatrix {
*+[l]{\hat X=\hat X^{(0)}}\ar[d]_{\tau}\ar[r]^{\varphi_1} &\hat X^{(1)}\ar[d]_{\tau_1}\ar[r]^{\varphi_{2}} &\cdots\ar[r]^{\varphi_{i}} &
\hat X^{(i)}\ar[d]_{\tau_i}\ar[r]^{\varphi_{i+1}} &\cdots\ar[r]^{\varphi_n} &\hat X^{(n)}\ar[d]_{\tau_n}\ar[dr]^\pi
\\
*+[l]{X=X^{(0)}} & X^{(1)} &\cdots & X^{(i)} &\cdots & X^{(n)} & Z
}}
\end{equation}
where $n\ge 0$, each $X^{(i)}$ is a Fano threefold with terminal Gorenstein singularities, $\tau_i$ is its small $\QQ$-factorialization
 \textup(or an isomorphism\textup),
$\hat X^{(i)}$ is a weak Fano threefold with terminal factorial singularities, and $\varphi_i:=X^{(i-1)}\to X^{(i)}$ is a birational Mori contraction which contracts $E_{i-1}\subset\hat X^{(i-1)}$ to a point or to a curve.
The morphism $\pi:\hat X^{(n)}\to Z$ is a Mori fiber space, where $Z$ is either a point, or a smooth rational curve, or a smooth del Pezzo surface. At each step of \eqref{eq:constr}
we have
\begin{equation}\label{eq:constr:rK}
\rk \Cl (X^{(i+1)})=\rk \Cl (X^{(i)})-1,\qquad (-K_{X^{(i+1)}})^3\ge (-K_{X^{(i)}})^3+2,
\end{equation}
where $\Cl(X^{(i)})$ is the Weil divisor class group of $X^{(i)}$.
We can choose the first $\QQ$-factorialization of $\tau:\hat X\to X$ and the sequence of contractions of $\varphi_i$ so that
$\hat X^{(n)}$ has no any divisorial $ K $-negative contractions, $\uprho(Z)\le2$, and
the variety $Z$ is either a point or isomorphic to $\PP^1$, $\PP^2$ or $\PP^1\times\PP^1$.
\end{theorem}
This construction was used also in \cite{Prokhorov-2005a} and \cite{Kaloghiros2011+err}.
For the convenience of the reader we present it here with complete proofs and with corrections
(see remark~\ref{sremark:Kal}).
Note that usually the sequence \eqref{eq:constr} is not unique.

\begin{proof}
Most of the assertions are inductive and can be proved by considering a single step $X^{(i-1)}\dashrightarrow X^{(i)}$.
Therefore, we consider only the first step $X=X^{(0)}\dashrightarrow X^{(1)}$.
The small $\QQ$-factorialization of $\tau:\hat X\to X$ exists according to \cite[Corollary 4.5]{Kawamata-1988-crep}.
Note that it is not unique and any two of them are different on the flops.
The variety $\hat X=\hat X^{(0)}$ has only terminal Gorenstein singularities and $K_{\hat X}=\tau^*K_X$.
Therefore, the divisor $-K_{\hat X}$ is nef and big, i.e. $\hat X$ is a weak Fano threefold.
Moreover, the linear system $|-nK_{\hat X}|$ for $n\gg 0$ defines a morphism that contracts only a finite number of curves.
The Mori cone of $\hat X$ is polyhedral:

\begin{slemma}[{\cite{Prokhorov-Shokurov-2009}}]
\label{lemma:FT}
Let $V$ be a projective variety with log terminal $\QQ$-factorial singularities such that the anticanonical class $-K_V$ is nef and big. Then the Mori cone $\NE(V)$ is polyhedral and generated by the classes of rational curves. For each extremal ray $R\subset\NE(V)$, there exists a contraction $\varphi_R:V\to Z$ which is
$K_V+\varepsilon\Delta$-negative for some effective $\QQ$-divisor $\Delta$ and $0 <\varepsilon\ll 1$.
\end{slemma}

There is an extremal $K$-negative Mori contraction $\varphi_1:\hat X^{(0)}\to\hat X^{(1)}$.
If it is not birational, then we put
$n=0$, $Z:=\hat X^{(1)}$, $\pi:=\varphi_1$. If, in addition, $Z$ is a curve, then it is rational because $H^1 (Z,\OOO_Z)=H^1 (X,\OOO_X)=0$.
If the base $Z$ is two-dimensional, then it is a smooth del Pezzo surface according to Corollary~\ref{cor:cb-fano}.

Now suppose that $\varphi_1:\hat X^{(0)}\to\hat X^{(1)}$ is birational.
Then the singularities of $\hat X^{(1)}$ are terminal and $\QQ$-factorial, and $\varphi_1$ contracts a Cartier divisor $E\subset\hat X$ (see Theorem~\ref{classification:extremal-rays}).
If $\hat X^{(1)}$ is not Gorenstein, then according to the classification~\ref{classification:extremal-rays}
we have $E\simeq\PP^2$ and $\OOO_E (-K_{\hat X})=\OOO_E (1)$, i.e. $\tau(E)$ is a surface of anticanonical degree $1$.
This contradicts our assumptions. Therefore, the variety $\hat X^{(1)}$ is Gorenstein.
Suppose that the divisor $-K_{\hat X^{(1)}}$ is not nef.
Since the singularities of $\hat X$ are terminal, according to \cite[Prop. 4.5 and Cor. 4.6]{Prokhorov-2005a}, there is only one possibility: $\varphi_1(E)$ is a smooth rational curve $B_1$ lying in the smooth locus of $\hat X^{(1)}$ and having normal bundle 
\[
\NNN_{B_1 /\hat X^{(1)}}\simeq\OOO_{\PP^1}(-1)\oplus\OOO_{\PP^1}(-2),
\]
and the variety $\hat X^{(1)}$ is smooth along $E$, where $E$ is a rational ruled surface $\FF_1$. Moreover, in this case, $(-K_{\hat X})^2\cdot E=1$ and $\tau$ contracts the negative section of $E\simeq\FF_1$ (see \cite[Cor. 4.7]{Prokhorov-2005a}). This again contradicts our assumptions.
Thus, the divisor $-K_{\hat X^{(1)}}$ is nef. It is clear that it is big, i.e. $\hat X^{(1)}$ is a weak Fano threefold.
Let $\tau_1:\hat X^{(1)}\to X^{(1)}$ be the morphism to the anticanonical model.
\begin{slemma}\label{slemma:tau}
The morphism $\tau_1$ does not contract any divisors and $(-K_{X^{(1)}})^2\cdot S_1\ge2$ for any surface $S_1\subset X^{(1)}$.
\end{slemma}
\begin{proof}
By the projection formula, it is easy to show that for any irreducible surface $\hat S\subset\hat X$, other than $E$, the inequality
\[
(-K_{\hat X^{(1)}})^2\cdot\varphi_1^*(\hat S)\ge (-K_{\hat X})^2\cdot\hat S=(-K_{X})^2\cdot\tau(\hat S)>1
\]
 holds (see, e.g., \cite[Lemma~2.5]{Prokhorov-v22}).
\end{proof}
It follows from Lemma \ref{slemma:tau} that the singularities of $X^{(1)}$ are terminal (and Gorenstein).
By construction, we have 
\[
\rk \Cl (\hat X^{(1)})=\uprho(\hat X^{(1)})=\uprho(\hat X)-1=\rk \Cl (X)-1.
\]
\begin{slemma}\label{lemma:constr:K}
$(-K_{\hat X^{(1)}})^3\ge (-K_{\hat X})^3+2$.
\end{slemma}
\begin{proof}
Consider, for example, the case where $\varphi_1$ is a contraction of type \type{B_1}. Then $\varphi_1(E)$ is a curve and $K_{\hat X}=\varphi_1^*K_{\hat X^{(1)}}+E$. We can write
\[
(-K_{\hat X})^3=(-K_{\hat X^{(1)}})^3 -3\varphi_1^*K_{\hat X^{(1)}}\cdot E^2-E^3,
\]
\[
2\varphi_1^*K_{\hat X^{(1)}}\cdot E^2+E^3=(-K_{\hat X})\cdot E>0,\quad
\varphi_1^*K_{\hat X^{(1)}}\cdot E^2=\varphi_1^*(-K_{\hat X^{(1)}})\cdot (-K_{\hat X})\cdot E\ge 0.
\]
Hence $(-K_{\hat X^{(1)}})^3>(-K_{\hat X})^3$. The required inequality now follows from the fact that both numbers $(-K_{\hat X^{(1)}})^3$ and $(-K_{\hat X})^3$ are even
(see Proposition~\ref{prop:g}). Cases of types \type{B_2}-\type{B_4} are treated in the same way.
\end{proof}
Lemma~\ref{lemma:constr:K} proves \eqref{eq:constr:rK}.
Thus, for $X^{(1)}$ all inductive assumptions are satisfied and we can continue the process \eqref{eq:constr}.
The process terminates quickly because the degree of Fano threefolds with canonical Gorenstein singularities is less or equal to 72 (see \cite{Prokhorov-2004a}).
Let us prove the last assertion. We may assume that $\hat X^{(n)}$ does not have divisorial extremal contractions.
Suppose that $Z$ is a del Pezzo surface different from $\PP^2$ and $\PP^1\times\PP^1$.
Then $Z$ contains a $(-1)$-curve.

\begin{slemma}\label{lemma:rho=3:negative}
Suppose that on $Z$ there exists a curve $C$ with negative self-intersection number.
Then $C$ is a $(-1)$-curve and there is a commutative diagram
\[
\xymatrix@R=11pt {
\hat X^{(n)}\ar[d]^{\pi}\ar@ {-->}[r]^\chi &\hat X'\ar[r]^{\varphi} &\bar X\ar[d]^{\bar\pi}
\\
Z\ar[rr]^{\sigma} &&\bar Z
}
\]
where $\sigma $ is a contraction of $C$, $\chi $ is a flop, $\varphi $ is a divisorial contraction, and $\bar\pi$ is a conic bundle.
\end{slemma}

\begin{proof}
Let $\sigma:Z\to\bar Z$ be a contraction of $C$. Then $\uprho(\hat X^{(n)} /\bar Z)=2$ and therefore the relative Mori cone $\NE(\hat X^{(n)} /\bar Z)$ is generated by two extremal rays. Both rays are contractible by Lemma~\ref{lemma:FT}.
The contraction of one of them is our conic bundle $\pi$.
The second contraction $\tau: \hat X^{(n)}\to X_\bullet$ 
cannot be a conic bundle (because in that case the generic fibers $\varphi$ and $\pi$ would coincide).
By our assumption the variety 
$\hat X^{(n)}$ has no divisorial extremal contractions.
 Hence, $\tau$ is a small contraction. By Theorem~\ref{classification:extremal-rays} it is $K$-trivial. 
Therefore, there exists a flop $\chi: \hat X^{(n)}\dashrightarrow\hat X'$. 
Again, the divisor $-K_{\hat X'}$ is nef and big, and $\uprho(\hat X' /\bar Z)=2$. So, on $\hat X'$ there is a $K_{\hat X'}$-negative contraction of $\varphi:\hat X'\to\bar X$ over $\bar Z$ which again cannot be a conic bundle. Therefore, $\varphi $ is divisorial and $\bar X$ is a threefold with factorial terminal singularities. Since $\uprho(\bar X /\bar Z)=1$, the projection $\bar\pi:\bar X\to\bar Z$ is a conic bundle.
\end{proof}

Going back to the proof of Theorem~\ref{theorem:constr}, let $A'$ be a very ample divisor on $\hat X'$ and let $A_n $ be its preimage on $\hat X^{(n)}$.
We put $A_i:=\varphi_{i+1}^*A_{i+1}$. Then $A_1$ is a big divisor on $\hat X$ such that the linear system $|A_1|$ has no fixed
components and defines a map 
\[
\Phi_{|A'|}:\hat X\dashrightarrow X'\subset\PP^{\dim|A'|}. 
\]
By Lemma~\ref{lemma:FT} we can run $A_1$-MMP on $\hat X$.
At each step, the exceptional locus is contained in the base locus of the proper transform of the linear system $|A_1|$.
In particular, it is one-dimensional. Hence, the corresponding birational transformation is the flop over $X$. After a finite number of flops, we get a new model
$\check\tau:\check X\to X$ such that the proper transform $\check A_1$ of the divisor $A_1$ is nef.
In this situation, the birational map $\check X\dashrightarrow\hat X'$ extends to a morphism $\check X\to\hat X'$ that can be decomposed
in the composition of extremal Mori contractions. Replace $\hat X$ with $\check X$ and $Z$ with $\bar Z$. This reduces the Picard number of the surface $Z$.
After a finite number of such transformations we achieve the situation where $Z$ does not contain $(-1)$-curves. Theorem~\ref{theorem:constr} is proved.
\end{proof}

\begin{sremark}
If, under the conditions of Theorem~\ref{theorem:constr}, the variety $X$ is smooth, then the sequence \eqref{eq:constr} can be chosen so that all the maps $\tau_i$ are isomorphisms and $\varphi_i$ are blowups of smooth curves, see \cite[(8.1)]{Mori-Mukai1983}.
\end{sremark}

\begin{sremark}
Under the conditions of Theorem~\ref{theorem:constr}, let $\iota(X)=2$. Then $\iota(X^{(i)})=2$ or $4$ for all $i$ and all morphisms of $\varphi_i$ are blowups of smooth points, see \cite{Prokhorov-GFano-1}.
\end{sremark}

\begin{sexample}[{\cite{Prokhorov-GFano-1}}]\label{example:constr}
Let $X^{(n)}=X^{(1)}:=\PP^1\times\PP^1\times\PP^1$ and let $X^{(1)}\subset\PP^7$ be the Segre embedding.
Let $X^{(1)}\dashrightarrow X=X^{(0)}\subset\PP^6$ be the projection from a point $P\in X^{(1)}$. This map is embedded in the diagram
of type \eqref{eq:constr}:
\[
\xymatrix@R=9pt {
\hat X\ar[d]^{\tau}\ar[r]^{\varphi_1} &\hat X^{(1)}\ar@ {=}[d]\ar[dr]^{\pi}
\\
X & X^{(1)}\ar@ {-->}[l] & *+[r]{Z=\PP^1\times\PP^1}
}
\]
Here $\varphi_1$ is the blowup of $P$ and the morphism $\tau$ contracts the proper transforms $\hat l_1,\,\hat l_2,\,\hat l_3$ of three lines $ l_i\subset X^{(1)}\subset\PP^7$
passing through the point $P$. Therefore, the singular locus of the variety $X$ consists of three ordinary double points and the anticanonical class of the variety $X$ is an ample Cartier divisor.
It is easy to see that the classes of curves $\hat l_1,\,\hat l_2,\,\hat l_3$ are linearly independent in $H_2 (\hat X,\QQ)$.
Hence, $\uprho(X)=1$.
\end{sexample}

\begin{sexample}[{\cite{Prokhorov-GFano-1}}]\label{example:constr1}
Let $X\subset\PP^4$ be the Segre cubic. 
This is a hypersurface of degree 3 having 10 singular points which are ordinary double.
The cubic $X$ is the half-anticanonical image of the blowup of $\PP^3$ at five general points (see e.g. \cite[\S~7]{Prokhorov-GFano-1}). Hence
there is a diagram of type \eqref{eq:constr}:
\[
\xymatrix@R=11pt@C=37pt{
\hat X\ar[d]^{\tau}\ar[r]^{\varphi_1} &\hat X^{(1)}\ar[d]^{\tau_1}\ar[r]^{\varphi_2} &\hat X_2\ar[d]^{\tau_2}\ar[r]^{\varphi_3} &\hat X_3\ar[d]^{\tau_3}\ar[r]^{\varphi_4}&\hat X_4\ar@{=}[d]\ar[r]^(.3){\varphi_5} & *+[l]{\hat X_5=\PP^3}\ar@{=}[d]\ar[dr]^{\pi}
\\
X & X^{(1)}& X_2& X_3& X_4&X_5&\mathrm{pt}
} 
\]
Here each $\varphi_i$ is the blowup of a point, $-K_{X^{(i)}}^3=8 (3+i)$ for all $i$ and $\uprho(X^{(i)})=1$ for  $i=0,\dots, 3$, and $\uprho(X_4)=2$.
\end{sexample}
 
\begin{sexample}[cf. \cite{Cheltsov:boundFano-en}]
Let $X^{(n)}$ be a Fano threefold with terminal Gorenstein singularities and let $P\in X^{(n)}$ be a smooth point.
Suppose that $X^{(n)}$ and $P$ satisfy the following conditions:
\begin{enumerate}
\renewcommand\labelenumi {\rm (\alph {enumi})}
\renewcommand\theenumi {\rm (\alph {enumi})}
\item
the anticanonical divisor is very ample and defines an embedding $X^{(n)}\subset\PP^{g+1}$;
\item
$X^{(n)}\subset\PP^{g+1}$ is an intersection of quadrics;
\item
the lines lying in $X^{(n)}$ do not pass through $P$ and the projection from the tangent space to $X^{(n)}$ to $P$ does not contract divisors.
\end{enumerate}
Let $\tau:\hat X^{(n)}\to X^{(n)}$ be a $\QQ$-factorization, let $\varphi_n:\hat X^{(n-1)}\to\hat X^{(n)}$ be the blowup of 
points $\tau^{-1}(P)$, and let $X^{(n-1)}$ be the anticanonical model $\hat X^{(n-1)}$.
Then $X^{(n-1)}$ is a Fano threefold with terminal Gorenstein singularities.
If the variety $X^{(n-1)}$ and some point on it satisfy the same conditions as $X^{(n)}$, then we can continue the process. We obtain the diagram
\eqref{eq:constr}, where each $\varphi_i$ is the blowup of a point.
\end{sexample}

\begin{proposition}\label{prop:iductive}
Notation as in Theorem~\xref{theorem:constr}.
\begin{enumerate}
\item\label{prop:iductive1}
If $\Bs|-K_X|=\varnothing$, then $\Bs|-K_{X^{(i)}}|=\varnothing$ for all $ i$.
\item\label{prop:iductive2}
If the divisor $-K_X$ is very ample, then the divisor $-K_{X^{(i)}}$ is also very ample for all $ i$.
\end{enumerate}
\end{proposition}

\begin{proof}
It follows from the proof of Theorems~\ref{theorem:Bs:hyp:trig-0} and~\ref{theorem:Bs:hyp:trig-1} that the linear system $|-K_{X^{(i)}}|$ is base point free (respectively, is very ample) if and only if
the image of the corresponding map $\Phi_{|-K_{X^{(i)}}|}$ is two-dimensional (respectively, $\Phi_{|-K_{X^{(i)}}|}$ is birational).
The assertions now follow from the fact that the maps $\tau_{i+1}\circ\varphi_{i+1}\circ\tau_i^{-1}:X^{(i)}\dashrightarrow X^{(i+1)}$ do not contract divisors.
\end{proof}

\begin{sremark}\label{sremark:Kal}
Suppose that in the construction \eqref{eq:constr} one has $\uprho(X)=1$. It was stated in \cite{Kaloghiros2011+err} and \cite{Kaloghiros2012} that $\uprho(X^{(i)})=1$ for all $i$.
Examples~\ref{example:constr} and~\ref{example:constr1} show that this is not true.
\end{sremark}

\section {Projective models}
The main results of this section are Theorems~\ref{theorem:Bs:hyp:trig-0},~\ref{theorem:Bs:hyp:trig-1}, and~\ref{theorem:Bs:hyp:trig-2}.
They slightly generalize the relevant facts for nonsingular Fano varieties {\cite[ch.~1, \S~6]{Iskovskikh-1980-Anticanonical}}, {\cite[\S~2.4]{IP99}}.
The ideas of proofs belong to V.A. Iskovskikh {\cite{Iskovskikh-1980-Anticanonical}} but since we are considering singular varieties, some changes have to be made.
On the other hand, there are much more general results for Fano varieties with canonical singularities
\cite{Jahnke-Radloff-2006}, \cite{Przhiyalkovskij-Cheltsov-Shramov-2005en}.
However, even in order to select varieties that meet our criteria from the lists \cite{Jahnke-Radloff-2006} and \cite{Przhiyalkovskij-Cheltsov-Shramov-2005en}, one
needs some effort.
We present relatively short proofs adapted to our situation.

\begin{theorem}\label{theorem:Bs:hyp:trig-0}
Let $X$ be a Fano threefold with terminal Gorenstein singularities.
Suppose that $\uprho(X)=1$ and $-K_X^3\ge 6$.
Then the linear system $|-K_X|$ is base point free.
\end{theorem}

\begin{proof}
Suppose that $\Bs|-K_X|\neq\varnothing$. Apply Theorem~\ref{thm:bs}. Since $-K_X^3\ge 6$,
we have the case~\ref{thm:bs}\ref{thm:bs1}. Let $g:=\g(X)$ and let $Z:=\Bs|-K_X|$ (a smooth rational curve).
According to Corollary~\ref{cor:bs} a general element $S\in|-K_X|$ is a smooth K3 surface.
By the Kawamata-Viehweg vanishing theorem $H^1 (X,\OOO_X)=0$. Hence, the restriction map
\[
H^1 (X,\OOO_X (-K_X))\longrightarrow H^1 (S,\OOO_S (-K_X))
\]
it is surjective and therefore $\Bs\bigl|-K_X|_S\bigr|=Z$.
According to \cite[Prop. 8.1]{Saint-Donat-1974} we have
\begin{equation}
\label{base-point-index=1}
\bigl|-K_X|_S\bigr|=Z+g|C|,\qquad g\ge 3,
\end{equation}
where $Z$ is a $(-2)$-curve on $S$ and
$|C|$ is a base point free elliptic pencil on $S$ such that
$Z\cdot C=1$. In particular,
\begin{equation*}
-K_X\cdot C=1\quad\text {and}\quad -K_X\cdot Z=g-2.
\end{equation*}
Consider the map $\Phi=\Phi_{|-K_X|}:X\dashrightarrow\PP^{g+1}$ defined by the linear system $|-K_X|$.
Its restriction $\Phi_S$ to the surface $S$ is the morphism
given by the linear system $|gC|$.
This linear system is composed of the elliptic pencil $|C|$ and one can easily check that 
it maps $S$ onto a rational curve $\Lambda\subset\PP^{g}$.
Moreover, $\Lambda=\Phi_S (Z)$ and so $\Lambda\subset\PP^{g}$
is a rational normal curve of degree $g$.
But this curve is nothing but a hyperplane section of $W=\Phi(S)\subset \PP^{g+1}$ corresponding to the divisor $S$.
Thus, $\dim W=2$ and $\deg W=g$. Therefore, the image of the map $\Phi_{|-K_X|}:X\dashrightarrow\PP^{g+1}$ is a surface
minimal degree
\begin{equation*}
W=W_g\subset\PP^{g+1}.
\end{equation*}
We obtain that $W$ is either a rational geometrically ruled surface $\FF_e$, or a cone over a rational normal curve of degree $g$, 
or a Veronese surface \cite[Th. 1.10]{Saint-Donat-1974}.
Indeterminacies of the map $\Phi$ are resolved by blowing up of the base locus $\Bs|-K_X|$, i.e. the curve $Z\simeq\PP^1$. We obtain the following diagram
\begin{equation*}
\xymatrix@R=10pt{
&\tilde X\ar[dl]_{\sigma}\ar[]!<0pt,0pt>;[dr]!<-40pt,5pt>^{\delta}&
\\
X\ar@{-->}[rr]^(.4){\Phi}&&W=W_g\subset\PP^{g+1}
} 
\end{equation*}
where $\sigma:\tilde {X}\to X$ is the blowup of
the curve $Z$ and the morphism $\delta$ is given by the anticanonical linear system $|-K_{\tilde X}|$. 
It follows from \eqref{base-point-index=1} that the generic fiber of the morphism $\delta$ is an elliptic curve and the exceptional divisor $E:=\sigma^{-1}(Z)$ is a section of this fibration. If $W$ is a cone, then the fiber over its vertex can be two-dimensional.

Since $1=\uprho(X)=\uprho(\tilde X)-1>\uprho(W)-1$, we have $\uprho(W)=1$.
This implies that $W\not\simeq\FF_e$.
Since $-K_{\tilde X}=\sigma^*(-K_X) -E $, the fibers of the ruled surface $\sigma_E:E\to Z$ are mapped to lines on $W$.
Therefore, $W$ cannot be a Veronese surface.
Thus, $W=W_g\subset\PP^{g+1}$ is a cone over a rational normal curve of degree $g$.
Let $L\subset W$ be a generator of the cone
and let $w_0\in W$ be its vertex.
Then $gL$ is a hyperplane section of the surface $W$ and therefore
$-K_{\tilde X}\sim\delta^*(gL)$. 
Clearly, $\tilde F:=\overline {\delta^{-1}(L\setminus\{w_0\})}$ is an irreducible surface and we can write 
\[
-K_{\tilde X}\sim\delta^*(gL)\sim g\tilde F+\tilde D, 
\]
where $\tilde D$ is an effective divisor with support in $\delta^{-1}(w_0)$. We put $F:=\sigma (\tilde F)$ and $D:=\sigma (\tilde D)$. Since the divisor $\tilde F$ is movable,
$F\neq 0$ and so
\begin{equation}
\label{eq:bs:KFD}
-K_X=-\sigma_* K_{\tilde X}=g F+D.
\end{equation}
If the divisor $D$ is Cartier, then so the divisor $F$ is (see Corollary~\ref{cor:Weil-Cartier}).
But then $\iota(X)\ge g>1$. This contradicts the assertion of Theorem~\ref{thm:bs}.
Therefore, we may assume that $D$ is not a Cartier divisor (in particular, $D\neq 0$).
Comparing \eqref{base-point-index=1} and \eqref{eq:bs:KFD} we see that $D\cap S=Z$.
Since $\uprho(X)=1$, the divisor $D$ is reduced and irreducible. 

Let $\tau:\hat{X}\to X$ be a small $\QQ$-factorialization \cite[Corollary 4.5]{Kawamata-1988-crep}.
Then $\hat{X}$ is a weak Fano threefold with factorial terminal singularities and
\begin{equation}
\label{eq:bs:KFDhat}
-K_{\hat{X}}=g\hat{F}+\hat{D},
\end{equation}
where $\hat{F}$ and $\hat{D}$ are proper transforms on $\hat{X}$ of divisors $F$ and $D$, respectively.
We can choose $\tau$ so that the divisor $-\hat{F}$ is relatively nef.
Then the divisor $\hat{D}$ is also relatively nef.
There is a $\hat{F}$-positive extremal ray on $\hat{X}$.
By our choice of $\tau$, it must be $K_{\hat{X}}$-negative.
Let $\varphi:\hat{X}\to Z$ be a contraction of this ray and let $ l $ be the corresponding minimal rational curve.

Suppose that $\hat D\cdot l <0$. Then the morphism $\varphi$ is birational and $\hat D$ is its exceptional divisor.
By~\eqref{eq:bs:KFDhat} 
\[
-(K_{\hat{X}}+\hat{D})\cdot l=g\hat{F}\cdot l\ge g\ge 4.
\]
This is impossible by Theorem~\ref{classification:extremal-rays}.

Therefore, $\hat{D}\cdot l\ge 0$.
Then, again by Theorem~\ref{classification:extremal-rays},
\[
3\ge -K_{\hat{X}}\cdot l=g\hat{F}\cdot l+\hat{D}\cdot l\ge g\ge 4.
\]
Again we have a contradiction.
Theorem~\ref{theorem:Bs:hyp:trig-0} is proved.
\end{proof}

Recall that a Fano threefold $X$ is called \emph{hyperelliptic} if its anticanonical linear system $|-K_X|$ defines a double cover onto its image.
A Fano threefold $X$ is called \emph{trigonal} if its anticanonical linear system $|-K_X|$ is very ample and defines an embedding to the projective space so that
its image is not an intersection of quadrics.

\begin{theorem}
\label{theorem:Bs:hyp:trig-1}
Let $X$ be a Fano threefold with terminal Gorenstein singularities and $\uprho(X)=1$.
We put $g:=\g(X)$.
If $g\ge 4$ and $\iota(X)=1$, then the linear system $|-K_X|$ is very ample and defines an embedding to $\PP^{g+1}$ except for the case described in Example~\xref{example:hyp:g=4} below.
\end{theorem}

\begin{sexample}
\label{example:hyp:g=4}
Consider the cone $W=W_3\subset\PP^5$ over a smooth surface $V=V_3\subset\PP^4$, $V\simeq\FF_1$ of degree $3$ with vertex $o\in W$. 
Fix two planes $\Pi_1$, $\Pi_2\subset W$ such that $\Pi_1\cap \Pi_2=\{o\}$.
Consider the hypersurface $Z=Z_4\subset\PP^5$ of degree $4$ passing through $\Pi_1$ and $\Pi_2$.
Let $B\subset W\cap Z$ be the residual subvariety to $\Pi_1\cup\Pi_2$.
Consider a double cover $X\to W$ branched over $B$. Then $X$ is a hyperelliptic Fano threefold of genus 4 and $X\to W=W_3\subset\PP^5$ is its anticanonical map,
see {\cite[2.7.3]{Takeuchi-2009}} and \cite[Theorem~1.6, case~$H_5$]{Przhiyalkovskij-Cheltsov-Shramov-2005en}.
For a general choice of the hypersurface $Z=Z_4\subset\PP^5$, the singular locus of the variety $X$ consists of a single ordinary double point, $\rk \Cl (X)=2$, and $\uprho(X)=1$. In this case, there is a Sarkisov link
\begin{equation*}
\vcenter {
\xymatrix@R=11pt {
&\hat{X}\ar[dr]^\tau\ar@ {-->}[rr]^{\chi}\ar[dl]_f &&\hat{X}'\ar[dr]^{f'}\ar[dl]_{\tau'}
\\
\PP^1 && X && Y'
}}
\end{equation*}
where $\tau$ and $\tau'$ are two small resolutions, $f$ is a del Pezzo fibration of degree 2, and $f'$ is the blowup of a factorial ordinary double point on a del Pezzo threefold $Y'$ degree 1 \cite[(2.7.3)]{Takeuchi-2009}, \cite[Lemma~3.4]{Grinenko2000}.
A general variety of this type is not rational \cite[Proposition~5.6]{Przhiyalkovskij-Cheltsov-Shramov-2005en}.
Note that this case is erroneously omitted in \cite{Jahnke-Peternell-Radloff-II}.
\end{sexample}

\begin{proof}[Proof of Theorem~\xref{theorem:Bs:hyp:trig-1}]
According to Theorem~\ref{theorem:Bs:hyp:trig-0} the linear system $|-K_X|$ defines a morphism
\[
\Phi=\Phi_{|-K_X|}:X\to W\subset\PP^{g+1},
\]
where $W:=\Phi(X)$.
Suppose that the morphism $\Phi$ is birational but not an embedding.
Let $S\in|-K_X|$ be a general element. Then the restriction of $\Phi$ to $S$ is birational.
By the Kawamata-Viehweg vanishing theorem of the restriction map of graded algebras
\begin{equation*}
\bigoplus_{n\ge 0} H^0 (X,\OOO_X (-nK_X)))\longrightarrow\bigoplus_{n\ge 0} H^0 (S,\OOO_S (-nK_X)))
\end{equation*}
is surjective on a homogeneous component of any degree.
According to the classical result on K3 surfaces \cite[Th. 6.1]{Saint-Donat-1974} the algebra on the right-hand side is generated by its component of degree 1.
It follows that the algebra on the left-hand side is also generated by its component of degree 1 (see \cite[Lemma 2.9]{Iskovskih1977a}).
In this case, $-K_X$ is a very ample divisor.

Suppose now that the morphism $\Phi$ is not birational to its image $W$.
We have
\begin{equation*}
2g-2=-K_X^3=(\deg\Phi)\cdot (\deg W).
\end{equation*}
Since $\deg\Phi\ge2$ and 
\[
\deg W\ge\operatorname {codim}(W)+1=g-1
\]
(according to the classical Enriques theorem, see \cite[\S 1]{Saint-Donat-1974}), we have the only possibility: $\deg\Phi=2$ and $\deg W=g-1$.
Thus, $W=W_{g-1}\subset\PP^{g+1}$ is a variety of minimal degree \cite[Th. 1.11]{Saint-Donat-1974}.
Since $\uprho(X)=1$, we have $\uprho(W)=1$. Since $g\ge 4$, the variety $W$ must be singular \cite[Th. 1.11]{Saint-Donat-1974}.

Consider the case $\dim\Sing (W)=1$. 
Then $\Sing (W)=L$ is a line and $W$ is a cone with vertex $L$ over a rational normal curve $C=C_{g-1}\subset\PP^{g-1}$ (see \cite[Th. 1.11]{Saint-Donat-1974}). 
It is easy to see that in this case the group $\Cl(W)$ is generated by the class of the plane $\Pi\subset W$. 
Since $\Phi$ is a finite morphism, there is a well-defined homomorphism $\Phi^*:\Cl(W)\to\Cl(X)$ which is an embedding. 
Moreover, $-K_X=\Phi^*H_W$, where $H_W$ is the class of the hyperplane section of $W$. 
A count of degrees gives us $H_W\sim (g-1)\Pi$. 
This means that $-K_X\sim (g-1)\Phi^*\Pi$, i.e. $\iota(X)\ge g-1\ge 2$ (see Corollary~\ref{cor:Weil-Cartier}). 
This contradicts our assumptions.

Finally, consider the case $\dim\Sing (W)=0$. Then $\Sing (W)=\{o\}$ is a point and $W$ is a cone with vertex $o$ over smooth
surface $V=V_{g-1}\subset\PP^{g}$ which is either a rational ruled surface or a Veronese surface
(again according to \cite[Th. 1.11]{Saint-Donat-1974}). In the second case, as above, we get $\iota(X)>1$. Consider the case where $V=V_{g-1}\subset\PP^{g}$ is a smooth rational ruled surface. In this case, the vertex of the cone $o\in W$ has a small resolution $\tilde W\to W$, where $\tilde W\simeq\FF (n_1, n_2, n_3)$, 
\[
n_3\ge n_2>n_1=0,\qquad g=\sum n_i+1\ge 4,
\]
and the morphism $\FF (0, n_2, n_3)\to W\subset\PP^{g-1}$ is given by the tautological linear system $|M|$.
Moreover, the minimal section $\tilde C\subset\tilde W$ of the projection $\FF (0, n_2, n_3)\to\PP^1$ is contracted to the vertex of the cone.
Since $o\in W$ is a non-Gorenstein singularity, the morphism $\Phi$ is ramified over $o$, i.e. $P:=\Phi^{-1}(o)$ is one point.
Consider the base change
\[
\xymatrix@R=7pt {
\tilde X\ar[r]^{\tilde\Phi}\ar[d] &\tilde W\ar[d]
\\
X\ar[r]^{\Phi} & W
}
\]
where $\tilde X$ is the normalization of $X\times_W\tilde W$. Here $\tilde X\to X$ is a small birational morphism.
Therefore, the variety $\tilde X$ has only terminal Gorenstein singularities. In particular, $\tilde X$ is smooth at the generic point of the fiber $\tilde\Phi^{-1}(\tilde C)$ over $\Phi^{-1}(o)$.
Let $B\subset W$ (resp. $\tilde B\subset\tilde W$) be the branch divisor of the morphism $\Phi$ (resp. $\tilde\Phi$). It is clear that $B$ is the image of $\tilde B$.
According to the Hurwitz formula we can write
\[
K_{\tilde X}=\Phi^*\left (K_{\tilde W}+\textstyle\frac12\tilde B\right)=-\Phi^*M,\quad\tilde B\sim 4M+2 (2-a_1-a_2) F,
\]
where $F$ is the class of a fiber of the projection $\FF (0, n_2, n_3)\to \PP^{1}$.
Since $\tilde X$ is smooth at a generic point of the curve $\tilde\Phi^{-1}(\tilde C)$, the branch divisor $\tilde B$ must be smooth
at a generic point of $\tilde C$.
Apply the following estimate of the multiplicity:
\begin{sproposition}[{\cite[Ch.~1, Prop.~3.13]{Iskovskikh-1980-Anticanonical}}]
Let 
\[
\EEE:=\oplus_{i=1}^m\OOO_{\PP^1}(n_i),\qquad 0\le n_1\le\cdots\le n_m,\quad n_1\neq n_m
\]
For each integer
$c$ put
\[
\EEE_c:=\oplus_{n_i\le c}\OOO_{\PP^1}(n_i).
\]
We identify $\PP_{\PP^1}(\EEE_c)$ with the subvariety in $\PP_{\PP^1}(\EEE)$, where the embedding defined by the projection
$\EEE\to\EEE_c$. 
Let $M$ be the tautological divisor on $\PP_{\PP^1}(\EEE)$ and let $F$ be the fiber of the projection $\PP_{\PP^1}(\EEE)\to\PP^1$. 
Then every section
\[
s\in H^0 (\PP_{\PP^1}(\EEE),\OOO (aM+bF))
\]
has a zero of order $\ge q$ along $\PP_{\PP^1}(\EEE_c)$ if and only if the following inequality is satisfied
\begin{equation}\label{eq:qe:scrolls}
ac+b+(n_m-c) (q-1) <0.
\end{equation}
\end{sproposition}
In our case, for $q=2$ and $c=n_1=0$, this gives us
\[
4n_1+4-2\sum n_i+n_3-n_1\ge 0.
\]
We obtain the only possibility: $n_2=1$, $n_3=2$, i.e. $\tilde W=\FF (0,1,2)$ and
$\tilde B\sim 4M-2F$.
Theorem~\ref{theorem:Bs:hyp:trig-1} is proved.
\end{proof}

\begin{theorem}
\label{theorem:Bs:hyp:trig-2}
Let $X$ be a Fano threefold with terminal Gorenstein singularities and $\uprho(X)=1$.
We put $g:=\g(X)$.
If $g\ge 5$ and $\iota(X)=1$, then the image $X=X_{2g-2}\subset\PP^{g+1}$ of the anticanonical embedding is an intersection of quadrics except for two cases described in Examples
\xref{example:trigonal:g=5} and~\xref{example:trigonal:g=6} below.
\end{theorem}

\begin{sexample}[{\cite[Th.~1.6, Case~$T_3$]{Przhiyalkovskij-Cheltsov-Shramov-2005en}}]
\label{example:trigonal:g=5}
Consider the cone $W=W_3\subset\PP^6$ over the Segre variety $V\subset\PP^5$, $V\simeq\PP^2\times\PP^1$.
Let $\Pi\subset W$ be a plane. 
Consider the cubic hypersurface $Z=Z_3\subset\PP^6$ passing through $\Pi$ and let $X=X_8\subset\PP^6$ be a residual subvariety of $W\cap Z$ to $\Pi$, that is, $W\cap Z=X\cup P$. Then $X$ is an anticanonically embedded Fano threefold of genus 5. 
Moreover, the quadrics in $\PP^6$ passing through $X$ cut out $W\subset\PP^6$. For a general choice of $Z$, the singular locus of $X$ consists of a single ordinary double point and $\uprho(X)=1$.
In this case, $X$ is included in the Sarkisov link of the form
\begin{equation}
\label{equation-diagram:a}
\vcenter {
\xymatrix@R=11pt {
&\hat{X}\ar[dr]^\tau\ar@ {-->}[rr]^{\chi}\ar[dl]_f &&\hat{X}'\ar[dr]^{f'}\ar[dl]_{\tau'}
\\
Y && X &&\PP^1
}}
\end{equation}
where $Y=\PP^2$, $f$ is a conic bundle with a discriminant curve of degree 7, and $f'$ is a fibration into cubic surfaces 
\cite[Table~2, Case~3]{BrownCortiZucconi-2004}, \cite[2.9.4]{Takeuchi-2009}, \cite[5.2]{Jahnke-Peternell-Radloff-II}.
A general variety from this family is not rational 
\cite[Th\'eor\`eme~4.9]{Beauville1977}, \cite[Claim~5.5]{Przhiyalkovskij-Cheltsov-Shramov-2005en}.
\end{sexample}

\begin{sexample}
\label{example:trigonal:g=6}
Consider the anticanonical image $Y=Y_{16}\subset\PP^{10}$ of a del Pezzo threefold of degree $2$ (double space branched in a quartic). Let $ l $ be a line on $Y$ (with respect to the standard polarization $-\frac 12 K_Y$). The image of $l$ on $Y_{16}\subset\PP^{10}$ is a conic. Let $X=X_{10}\subset\PP^{7}$ be the image of $Y_{16}\subset\PP^{10}$ when projected from $ l$. Then $X=X_{10}\subset\PP^{7}$ is a trigonal Fano threefold with terminal Gorenstein singularities \cite[example 1.11, remark 4.17]{Przhiyalkovskij-Cheltsov-Shramov-2005en}.
As above, in this case a general variety $X$ of this type is included in the Sarkisov link \eqref{equation-diagram:a},
where $Y$ is a del Pezzo threefold of degree 2, $f$ is the blowup of a line on it, and $f'$ is a fibration into cubic surfaces 
\cite[2.9.3]{Takeuchi-2009}, \cite[6.5 ]{Jahnke-Peternell-Radloff-II}.
A general variety from this family is not rational \cite{Voisin1988}.
\end{sexample}

\begin{proof}[Proof of Theorem~\xref{theorem:Bs:hyp:trig-2}]
Let $W\subset\PP^{g+1}$ be an intersection (as a scheme) of all quadrics passing through $X$.
Fix a point $P\in W$ which is not singular for $X$ and consider a sufficiently general
subspace in
$\PP^{g}\subset\PP^{g+1}$ such that $S:=X\cap\PP^{g}$ is a smooth surface (of type K3).
It can be shown (see \cite[Lemma 7.9]{Saint-Donat-1974}) that
any quadric $Q\subset\PP^{g}$ passing through $S$ extends to a quadric
$Q'\subset\PP^{g+1}$ passing through $X$.
Therefore, the scheme $W_S:=W\cap\PP^{g}$
is cut out by quadrics in $\PP^{g}$ containing $S$ and by virtue of \cite[\S 7]{Saint-Donat-1974}
the variety $W_S$ is irreducible and it is a three-dimensional subvariety $W_S\subset\PP^{g}$ of minimal degree $g-2$.
Therefore, $W$ is also irreducible and it is a variety of minimal degree $W=W_{g-2}\subset\PP^{g+1}$.
Since $g\ge 5$, again according to \cite[\S 7]{Saint-Donat-1974},
the variety $W_S$ is smooth at $P$ and the same is true for $W$.
Therefore, $W$ is smooth outside the singular points of $X$:\quad $\Sing (W)\subset\Sing (X)$.
If $W$ is smooth everywhere, then according to the classical Enriques theorem \cite[Th. 1.11]{Saint-Donat-1974} the variety $W=W_{g-2}\subset\PP^{g+1}$ has a structure of a $\PP^3$-bundle over $\PP^1$.
So, in this case, $X$ is surjectively mapped onto the curve.
This contradicts our assumption $\uprho(X)=1$. 
Since the variety $X$ has at worst isolated singularities, $W=W_{g-2}\subset\PP^{g+1}$ is a cone with a zero-dimensional vertex $w_0$ over a smooth three-dimensional variety $V=V_{g-2}\subset\PP^{g}$ and $X$ is singular at $w_0$. The vertex of the cone $w_0\in W$ has a resolution $\sigma:\hat W\to W$, where $\hat W\simeq\PP_{\PP^1}(\EEE)$ and $\EEE $ is a rank $4$ vector bundle  on $\PP^1$. 
Moreover, the exceptional locus $\hat C:=\sigma^{-1}(w_0)$ is a smooth rational curve, the section of the bundle $\pi:\PP_{\PP^1}(\EEE)\to\PP^1$. 
Let $\hat X\subset\hat W$ be the proper transform of $X$ and let $\eta:\hat X\to X$ be the restriction of $\sigma $ to $\hat X$.
We may assume that $\EEE=\oplus_{i=1}^4\OOO_{\PP^1}(n_i)$, where
\[
\sum n_i=g-2\ge 3,\qquad n:=n_4\ge n_3\ge n_2>n_1=0,
\]
and the morphism $\hat W\to W\subset\PP^{g+1}$ is given by the linear system $|M|$, where $M$ is the tautological divisor on $\PP_{\PP^1}(\EEE)$. Let $F$ be the fiber of the projection $\pi:\hat W\simeq\PP_{\PP^1}(\EEE)\to\PP^1$. Then
\[
-K_{\hat W}=4M+\left (2-\textstyle\sum n_i\right) F=4M+(4-g) F,\qquad
-K_{\hat X}=M|_{\hat X}.
\]
Thus, $K_{\hat X}=\eta^*K_X$, i.e. the morphism $\eta:\hat X\to X$ is crepant and therefore the variety $\hat X$ has at worst terminal Gorenstein singularities.
According to the adjunction formula,
\[
\hat X\sim 3M+(4-g) F=3M+\left (2-\textstyle\sum n_i\right)F.
\]
In particular, $\hat X$ is nonsingular at the generic point of $\hat C$. Apply the estimate of the multiplicity \eqref{eq:qe:scrolls} with $q=2$ and $c=n_1=0$.
We obtain $n_2+n_3\le2$, i.e. $n_2=n_3=1$. Further, the three-dimensional subvariety $\PP_{\PP^1}(\EEE')$ corresponding to the surjection
\[
\EEE\to\EEE'=\OOO_{\PP^1}(n_2)\oplus\OOO_{\PP^1}(n_1)\oplus\OOO_{\PP^1},
\]
not a component of $\hat X$. Using \eqref{eq:qe:scrolls} again, we obtain
\[
3n_3+2-\sum n_i\ge 0,\qquad n_4\le 3.
\]
Thus, $n_2=n_3=1\le n\le 3$, i.e. 
\[
\EEE\simeq\OOO_{\PP^1}\oplus\OOO_{\PP^1}(1)\oplus\OOO_{\PP^1}(1)\oplus\OOO_{\PP^1}(n),\quad n=1,\, 2,\, 3.
\]
Therefore,
\[
g=n+4,\qquad g=5,\, 6,\, 7.
\]
Note also that $W$ is a cone over a smooth rational scroll $V\subset\PP^{n+4}$, $V\simeq\FF (1,1, n)$.

Consider the case $n=1$ (i.e. $g=5$). Then $V\simeq\PP^1\times\PP^2$ (with Segre embedding to $\PP^5$) and $W=W_3\subset\PP^6$
is a cone over $V$. Since $\hat X\sim 3M-F$, there exists a cubic hypersurface $Z=Z_3\subset\PP^6$ passing through $X$ and the residual subvariety in $Z\cap W$ is the plane $\sigma (F)$. We obtain the variety from Example~\ref{example:trigonal:g=5}.
Similarly, in the case of $n=2$ we get the variety from Example~\ref{example:trigonal:g=6}.

Let $n=3$. The linear system $|M-nF|$ contains a unique irreducible divisor $\hat D$.
This divisor is a rational scroll $\FF (1,1,3)$ and contains the section $\hat C$.
Using \eqref{eq:qe:scrolls} on $\hat D$ again, we get that the multiplicity of $\hat D\cap\hat X$ along $\hat C$ equals $3$.
This means that the divisor $\hat D\cap\hat X$ is reducible: $\hat D\cap\hat X=R_1+R_2+R_3$, where the intersection of $R_i$ with a general fiber $F$ of the projection $\pi$ is a line.
For a general choice $\hat X\in|3M-3F|$, the variety $\hat X$ is nonsingular and $X$ is a trigonal Fano threefold with an ordinary double point in $\sigma (\tilde C)$
(see \cite[case $T_{16}$]{Przhiyalkovskij-Cheltsov-Shramov-2005en}). In this case, $R_i$ are Cartier divisors and their some linear combination $\sum a_i R_i$
trivially intersects $\hat C$. But then $\sigma_* (\sum a_i R_i)$ is a Cartier divisor on $X$ which is not proportional to $K_X$. This contradicts our assumption $\uprho(X)=1$.
Theorem~\ref{theorem:Bs:hyp:trig-2} is proved.
\end{proof}

\subsection{Complement}
Consider the construction of trigonal varieties from the proof of Theorem~\ref{theorem:Bs:hyp:trig-2} in more detail.
The base locus of the linear system $|N|:=|M-F|$ consists exactly of the curve $\hat C$ and a general element $|N|$ is irreducible. Let $\gamma:\tilde W\to\hat W$ be the blowup of $\hat C$. The composition $\tilde W\longrightarrow\hat W\longrightarrow W$ is the blowup of the vertex of the cone $w_0\in W$, the natural projection of $\tilde W\to V$ is a $\PP^1$-bundle and the exceptional divisor $E :=\gamma^{-1}(\hat C)$ is its section.
Let $M^*:=\gamma^*M$ and $F^*:=\gamma^*F$.
Then the linear system $|\tilde N|:=|M^*- F^*- E|$ is base point free and defines a morphism $\phi:\tilde W\to\PP^{n+1}$. It is easy to compute:
\[
(M^*)^4=n+1,\ (M^*)^3\cdot F^*=1,\ M^*\cdot E^3=0,\ F^*\cdot E^3=1,\ E^4=-n-2.
\]
Since $\tilde N^3\cdot F^*=0$, the morphism $\phi$ contracts the divisors $F^*$: the image of each $F^*$ is a plane $\PP^2$ on $\phi(\tilde W)$ and in fibers $F^*\to\phi(F^*)$
are lines which are fibers of $\tilde W / V$. Hence we conclude that $\phi(\tilde W)=\phi(E)$ and $\phi(\tilde W)=\PP^{n+1}$ for $n=1$ and $ 2$. 
Since $\tilde N^3\cdot E=n-1$, the morphism $\phi_E:E\to\phi(E)$ is birational for $n=2$ and $3$, and for $n=3$ its image $\phi(\tilde W)$ is a quadric in $\PP^4$.
We obtain a commutative diagram
\begin{equation}
\label{eq:diag:trig}
\vcenter{
\xymatrix@R=11pt{
&&
\\
W &\hat W\ar[l]_{\sigma}\ar[d]^{\pi} &\tilde W\ar@/_1.5pc/[ll]\ar[d]\ar[rr]^(.3){\phi}\ar[l]_{\gamma}&&\phi(\tilde W)\subset\PP^{n+1}
\\
&\PP^1 &V\ar[l]\ar[rru]&
}}
\end{equation} 
Note that $E\simeq V\simeq\FF [1,1, n]$. We also have
\[
\tilde N^3\cdot\tilde X=2n-2.
\]

Consider the case $n=1$. Then $E\simeq\PP^1\times\PP^2$ and its normal bundle has the form $\NNN_{E /\tilde W}\simeq\OOO_{\PP^1}(-1)\boxtimes\OOO_{\PP^2}(-1)$. Therefore, there are two contractions $\tilde W\to\hat W$ and $\tilde W\to\hat W'$ of the divisor $E$ in different directions. We obtain that $\phi:\tilde W\to\PP^2$ decomposes into the composition $\phi:\tilde W\longrightarrow\hat W'\longrightarrow\PP^2$.
Here, the projection of $\hat W'\to\PP^2$ is a $\PP^1$-bundle and
induced rational map $\hat W\dashrightarrow\hat W'$ is the Kawamata flip \cite{Kawamata1989}.

For $n=2$ and $3$, the anticanonical divisor $-K_{\tilde W}=2\tilde N+2M^*- (n-2) F^*$ is relatively nef over $\phi(\tilde W)$ and the linear system $|M-nF|$ contains a unique irreducible divisor $D$, having multiplicity 1 along $\hat C$.
Let $\tilde D\subset\tilde W$ be the proper transform of $D$. It is easy to compute that $\tilde N^3\cdot\tilde D=\tilde N^2\cdot\tilde D^2=\tilde N\cdot\tilde D^3=0$.
Therefore, the divisor $\tilde D$ is contracted to a rational curve.
We also have
\[
\tilde N^2\cdot\tilde X\cdot\tilde D=0.
\]

Consider the case of $g=6$ (that is, $n=2$). Then $W$ is a cone over a smooth rational scroll $V\subset\PP^{6}$, $V\simeq\FF [1,1,2]$.
Moreover, $V\to\PP^3$ is the blowup of a line.
Since $\tilde N^3\cdot\tilde X=2$ and $\tilde N^2\cdot\tilde X\cdot E=0$, the induced map $\hat X\dashrightarrow\PP^3$ is double cover at a general point and does not contract any divisors.
Top line in \eqref{eq:diag:trig} induces maps on $X$:
\[
X\longleftarrow\hat X\longleftarrow\tilde X\longrightarrow V\longrightarrow\PP^3,
\]
where $\tilde X\to V\to\PP^3$ is the Stein factorization for $\phi|_{\tilde X}$.
Here $V\to\PP^3$ is a double cover and $V$ is a Fano threefold, double space branched in a quartic.
We obtain the variety from Example~\ref{example:trigonal:g=6}.

Finally, consider the case of $g=7$ (that is, $n=3$). Then the variety $W$ is a cone over a smooth rational scroll $V\subset\PP^{7}$, $V\simeq\FF [1,1,3]$.
The image of $\phi(\tilde W)$ is a quadric $Q\subset\PP^4$, singular along the line $l\subset Q$.
Since $\tilde N^3\cdot\tilde D=\tilde N^2\cdot\tilde D^2=0$, the morphism $\phi$ contracts the divisor $\tilde D$ to the line $l$.
Since $\tilde N^3\cdot\tilde X=4$ and $\tilde N^2\cdot\tilde X\cdot E=0$, the induced map $\hat X\dashrightarrow Q$ is double cover at a general point.

\section {Fano threefolds containing a plane}
\begin{pusto}
\label{notation_1}
Let $X=X_{2g-2}\subset\PP^{g+1}$ be a Fano threefold with terminal Gorenstein singularities
such that $\Pic(X)=\ZZ\cdot K_X$ and $g\ge 5$. In this section we assume that
$X$ is not trigonal and contains a plane.
\end{pusto}

\begin{theorem}
\label{th:plane}
Let $X=X_{2g-2}\subset\PP^{g+1}$ be an anticanonically embedded \textup(see Theorem~\xref{theorem:Bs:hyp:trig-1}\textup)
Fano threefold with terminal Gorenstein singularities
such that $\Pic(X)=\ZZ\cdot K_X$ and $g\ge 5$.
Suppose that $X$ is not trigonal and contains the plane $\Pi$. If the variety
$X$ is not rational, then $\g(X)=6$ and $X$ is birationally equivalent to a smooth three-dimensional cubic $Y_3\subset\PP^4$ and  the corresponding map 
$Y_3\dashrightarrow X_{2g-2}\subset\PP^{g+1}$ is given by the linear system of quadrics passing through a connected reduced curve $\Gamma\subset Y_3\subset\PP^4$ of degree $3$
of arithmetic genus $0$ such that $\dim\langle\Gamma\rangle=3$.
In this case, the plane $\Pi$ is an image of the cubic surface $\Pi_Y:=Y\cap\langle\Gamma\rangle$.
For the curve $\Gamma$, only the following possibilities occur \textup(see examples below\textup):
\begin{enumerate}
\item\label{th:plane1}
$\Gamma$ is a smooth rational twisted cubic curve;
\item\label{th:plane2}
$\Gamma=\Gamma_1\cup\Gamma_2$, where $\Gamma_2$ is a non-degenerate conic and $\Gamma_1$ is a line;
\item\label{th:plane3}
$\Gamma=\Gamma_1\cup\Gamma_2\cup\Gamma_3$ is a chain of three lines;
\item\label{th:plane4}
$\Gamma=\Gamma_1\cup\Gamma_2\cup\Gamma_3$ is the union of three lines passing through one point.
\end{enumerate}
\end{theorem}

Let us give some examples.

\begin{sexample}\label{example:plane}
Let $Y=Y_3\subset\PP^4$ be a smooth three-dimensional cubic and let $\Gamma\subset Y$ be a
twisted rational cubic curve. Consider the blowup
$f:\hat{X}\to Y$ of $\Gamma$.
Let $\hat D\subset\hat{X}$ be an exceptional divisor,
let $L$ be a hyperplane section of $Y$, and let $L^*:=f^*L$.
The linear system $|-K_{\hat{X}}|=|2L^*-D|$ is base point free and defines a
morphism $\tau:\hat{X}\to X=X_{10}\subset\PP^7$ to a Fano threefold of genus $6$.
Let $\Pi_Y\subset Y$ be the hyperplane section passing through $\Gamma$.
Then for its preimage $\Pi_{\hat{X}}\subset\hat{X}$ we have $\Pi_{\hat{X}}\sim L^*- D$.
This implies 
\[
\deg\tau(\Pi_{\hat{X}})=(-K_{\hat{X}})^2\cdot (L^*- D)=1, 
\]
i.e. $\Pi:=\tau(\Pi_{\hat{X}})$
is a plane. All curves contracted by the morphism $\tau$ lie on the surface $\Pi_{\hat{X}}$ (which is not contracted).
Therefore, $\tau$ is a small morphism and the singularities of $X$ are terminal.
Moreover, $\Pi_{\hat{X}}$ is a weak del Pezzo surface of degree 3 and the restriction of $\tau_\Pi:\Pi_{\hat{X}}\to\Pi$ is the blowup of six points in almost general position.
It is also clear that $\uprho(X)<\uprho(\hat X)=2$, i.e. $\uprho(X)=1$.
The construction can be completed to the following Sarkisov link 
\begin{equation*}
\vcenter {
\xymatrix@R=11pt {
&\hat{X}\ar[dr]^\tau\ar@ {-->}[rr]^{\chi}\ar[dl]_f &&\hat{X}'\ar[dr]^{f'}\ar[dl]_{\tau'}
\\
Y && X && Y'
}}
\end{equation*}
where $\chi$ is a flop and $f'$ is an extremal Mori contraction
which contracts its preimage $\Pi_{\hat{X}}$ to a non-Gorenstein point
of type $\frac 12 (1,1,1)$ \cite{BlancLamy:cubic}, \cite{Cutrone-Marshburn}.
\end{sexample}
The following example is a degeneration of Example~\ref{example:plane}.
\begin{sexample}\label{example:plane:a}
Let $Y=Y_3\subset\PP^4$ be a smooth three-dimensional cubic and let $\Gamma\subset Y$ be a connected
reducible curve of degree 3 of arithmetic genus 0 having ordinary double points.
Thus, the curve $\Gamma$ is of type~\ref{th:plane}\ref{th:plane2} or~\ref{th:plane3}.
Let be
$f:\hat{X}\to Y$ be the the blowup of $\Gamma$. The variety $\hat X$ has only terminal Gorenstein (non-factorial) singularities over $\Sing (\Gamma)$.
As above, the linear system $|-K_{\hat{X}}|$ defines
a morphism $\tau:\hat{X}\to X=X_{10}\subset\PP^7$ to a Fano threefold of genus $6$ with terminal Gorenstein singularities and the image of a hyperplane section passing through $\Gamma$ is a plane on $X$.
\end{sexample}

\begin{sexample}\label{example:plane1}
Let $Y=Y_3\subset\PP^4$ be a smooth three-dimensional cubic of the type~\ref{th:plane}\ref{th:plane4}, i.e. let $\Gamma=\Gamma_1\cup\Gamma_2\cup\Gamma_3\subset Y$ be the union of three lines passing through one point $P\in X$ and not lying in the same plane.
Consider the symbolic blowup
$f_1:\hat{X}_1\to Y$ of $\Gamma$. Recall that this is the relative $\operatorname{\mathbf{Proj}}$ of the sheaf of algebras $\oplus\mathcal {I}_{\Gamma}^{[n]}$,
where $\mathcal {I}_{\Gamma}^{[n]}$ is the ideal in $\OOO_{Y}$ consisting of elements that vanish with the multiplicity $\ge n $ along $\Gamma$.
It can be shown that in this situation $\hat X_1$ has a unique singular point that is terminal of type $\frac12 (1,1,1)$ \cite{Prokhorov-Reid}. 
Let $\hat X\to\hat{X}_1$ be the blowup of this point and let $\hat E_1$ be its exceptional divisor. Then $\hat X\to\hat{X}_1$ is an extremal contraction of type \type{B_5}. The linear system $|-K_{\hat{X}}|$ defines a morphism $\tau:\hat{X}\to X=X_{10}\subset\PP^7$ to a Fano threefold of genus $6$ with Gorenstein terminal singularities and the image of $\hat E_1$ is a plane on $X$.
\end{sexample}

\begin{proof}[Proof of Theorem~\xref{th:plane}]
Put $\LLL:=|-K_X|$. Then $\dim\LLL=g+1\ge 6$. Consider also the following linear system $\HHH\subset\LLL$:
\[
\HHH:=\{H\mid H+\Pi\in\LLL\}.
\]
Clearly, $\HHH$ has no fixed components and
\[
\dim\HHH=\dim\LLL-3=g-2\ge 3.
\]
Consider the map
\[
\Phi=\Phi_{\HHH}:X\dashrightarrow\PP^{g-2}
\]
given by the linear system $\HHH$, i.e. the projection from $\Pi$.
We put $Y:=\Phi(X)$.

\begin{lemma}
\label{lemma-predv-image}
The map $\Phi$ is birational onto its image and $g\ge 6$.
\end{lemma}
\begin{proof}
The general fiber $\Phi^{-1}(y)$ of the map $\Phi:X\dashrightarrow Y$ is the residual subvariety of the intersection of
$X$ and the subspace $\PP^3$ generated by $\Pi$ and the point $y$. The set $X\cap\PP^3$
is an intersection of quadrics in $\PP^3$ and contains $\Pi$.
So $X\cap\PP^3\setminus\Pi$ is a linear subspace in $\PP^3$.
If this is a point, then $\Phi$ is birational. If $X\cap\PP^3\setminus\Pi$ is a plane, then
$Y$ is a rational curve and $\Phi$ is a birationally $\PP^2$-bundle. Hence, the variety $X$ is rational in this case. This contradicts our assumptions.
If $X\cap\PP^3\setminus\Pi$ is direct, then
$Y$ is a rational surface and $\Phi$ is a birationally $\PP^1$-bundle. Again $X$ is rational, a contradiction.
Thus, the map $\Phi$ is birational.
If $g=5$, then $Y=\PP^3$ and therefore the variety $X$ is also rational. This again contradicts our assumptions.
\end{proof}

Let $\sigma:\tilde X\to X$ be the blowup of $\Pi$. This is a small birational morphism that can be extended to $\QQ$-factorialization $\hat X\to\tilde X\to X$.
Since the linear system $\HHH_{\hat{X}}$ is base point free,
the pair $(\hat{X},\HHH_{\hat{X}})$ is terminal.
We have
\begin{equation*}
K_{\hat{X}}+\HHH_{\hat{X}}+\Pi_{\hat{X}}\sim 0.
\end{equation*}
Let us run the MMP with respect to the log divisor $K_{\hat X}+\hat\HHH$, first over $Y$, and then over $\operatorname{\mathbf{Spec}}(\CC)$.
We obtain the following diagram:
\[
\xymatrix@C=50pt@R=15pt{
\tilde X\ar[d]_{\sigma}\ar[drr]&\hat{X}\ar[l]\ar@{-->}[r]^{\text{MMP}/ Y}\ar[dr]^{f}&V\ar[d]^{\varphi}\ar@{-->}[r]^{\text{MMP}}& U\ar[d]^{\upsilon}
\\
X\ar@{-->}[rr]^{\Phi}& &Y&Z
} 
\]
We have
\begin{equation}
\label{eq-HB}
K_V+\HHH_V+\Pi_V\sim 0.
\end{equation}
Note that $\HHH_V$ is the proper transform of the linear system of
hyperplane sections of the variety $Y$.
In particular, it follows that $\HHH_V$ is base point free and
$\dim\HHH_V=g-2\ge 4$.
The variety $U$ has a structure of a Mori fiber space $\upsilon:U\to Z$.
If $\dim Z=2$, then $\upsilon$ is a rational curve fibration over a rational surface $Z$.
Since the image of the map defined by the linear system $\HHH_{U}$ is
three-dimensional, $\HHH_{U}$ is not a preimage of a linear system on $Z$, i.e. 
$\HHH_{U}$ is $\upsilon$-horizontal. For a fiber $l$ of the contraction $\upsilon$ over a sufficiently general point
we have $ U\simeq\PP^1$, $K_U\cdot l=-2$, and  $\HHH_U\cdot l>0$.
Therefore, $\HHH_U\cdot l=1$.
Thus, $\upsilon$ is a rational curve fibration  with a birational section $\HHH_U$.
In this case, $X$ must be rational.
If $\dim Z=1$, then $Z\simeq\PP^1$ and the generic fiber $\upsilon $ is either a two-dimensional quadric or
by plane. As above, $X$ is rational.

Thus, $Z$ is a point.
In this case, $U$ is a Fano threefold with terminal $\QQ$-factorial singularities
and Picard number $\uprho(U)=1$. Moreover, there is a decomposition
\begin{equation}
\label{eq-HB1}
K_U+\HHH_U+\Pi_U\sim 0,
\end{equation}
where $\Pi_U$ is a prime Weil divisor, $\HHH_U$ is a linear system of Weil divisors without
fixed components, $\dim\HHH_U=g-2\ge 4$, and the pair $(U,\HHH_U)$ is terminal.
In particular, the generic element $H_U\in\HHH_U$ is nonsingular, lies in the smooth locus of $U$, and
$\Bs|\HHH_U|\le 0$.
By the adjunction formula $-K_{H_U}\equiv\Pi_U|_{H_U}$ and this divisor is ample.
So $H_U$ is a smooth del Pezzo surface. In particular, it is
rational.
By our assumption, the variety $U$ is not rational.
Such pairs  $(U, H_U)$ are classified (see \cite[Theorem 1.5]{CampanaFlenner1993}).
Taking the inequality $\dim|H_U|\ge\dim\HHH_U\ge 4$ into account we obtain that $ U=U_3\subset\PP^4$
and $\HHH_U$ is a complete linear system of hyperplane sections and $\Pi_U$ is also a hyperplane section. In particular, $\HHH_U$
is base point free. This implies that the map $\hat{X}\dashrightarrow U$ is defined by the
linear system $\HHH_{\hat{X}}$. Therefore, $U=Y=V$ and this map is a morphism.
We obtain that $g=6$ and the linear system $\HHH_U$ is base point free.
Let $\hat E=\sum\hat E_i$ be the exceptional divisor of  $f$ and let $\Gamma\subset Y$ be the union of the one-dimensional components of $ f(E)\subset Y$.
The curve $\Gamma$is contained in the surface $\Pi_Y$ which is a hyperplane section of the cubic $Y$.
In particular, $\dim \langle\Gamma\rangle\le 3$.
Further,
\[
K_{\hat{X}}^2\cdot\HHH_{\hat{X}}=(f^*K_U+E)^2\cdot f^*\HHH_{U}=K_U^2\cdot\HHH_{U} -\HHH_{U}\cdot\Gamma.
\]
On the other hand,
\[
K_{\hat{X}}^2\cdot\HHH_{\hat{X}}=-K_{\hat{X}}^3 -K_{\hat{X}}^2\cdot\Pi_{\hat{X}}=10-1=9.
\]
Hence, $\HHH_{U}\cdot\Gamma=K_U^2\cdot\HHH_{U} -9=12-9=3$. Thus, $\Gamma\subset Y$ is a curve of degree $ 3$.
Note that $\LLL_Y=f_*\hat\LLL$ is a linear system of quadratic sections of the cubic $Y$, $\dim\LLL_Y=\dim\LLL=7$, and $\Bs\LLL_Y\supset\Gamma$. This linear system defines a map $Y\dashrightarrow X=X_{10}\subset\PP^7$.

We claim that the curve $\Gamma$ is connected. This is obvious if $\dim\langle\Gamma\rangle=2$.
Let $\dim\langle\Gamma\rangle=3$. Then $\Gamma$ is a curve of degree 3 in $\PP^3$, through which at least two different quadrics pass.
Therefore, $\Gamma$ is contained in the singular quadric $Q\subset\PP^3$. If $Q$ is irreducible, then $\Pic(Q)\simeq\ZZ$ and the curve $\Gamma$ is obviously connected.
Hence, we may assume that $Q$ is the union of two planes. Then
$\Gamma$ by the union of a conic lying in one plane and a line lying in another one.
It is easy to see that such a curve $\Gamma$ is connected.
Therefore, for $\Gamma$ there are only possibilities~\ref{th:plane}\ref{th:plane1}-\ref{th:plane4} and possibly another additional possibility
\begin{enumerate}
\setcounter{enumi}{4}
\item\label{th:plane5}
$\dim\langle\Gamma\rangle=2$ and
$\Gamma=Y\cap\langle\Gamma\rangle$, i.e. $\Gamma$ is a plane cubic.
\end{enumerate}
In all cases, the curve $\Gamma$ is cut out on $Y$ by the linear system $\LLL_{Y,\Gamma}$ of quadratic sections passing through it.
According to the above $\LLL_{Y}\subset\LLL_{Y,\Gamma}$ and $\Gamma=\Bs\LLL_{Y,\Gamma}$ (as a scheme).
It is also easy to compute that $m:=\dim\LLL_{Y,\Gamma}=7+\p (\Gamma)$.
In the cases~\ref{th:plane}\ref{th:plane1} -\ref{th:plane4} we have $\LLL_Y=\LLL_{Y,\Gamma}$ and therefore the variety $X$ coincides with the image of the map
$\Phi_{\LLL_{Y,\Gamma}}:Y\dashrightarrow X=X_{10}\subset\PP^{7}$ which should be birational.
The existence of the corresponding $X$ follows from Examples~\ref{example:plane} -\ref{example:plane1}.

Consider the case~\ref{th:plane5}. Then
$\LLL_Y$ has codimension 1 in $\LLL_{Y,\Gamma}$.
This implies that the linear system $\LLL_{Y,\Gamma}$ defines a map $\Phi_{\LLL_{Y,\Gamma}}:Y\dashrightarrow X^{\sharp}\subset\PP^{8}$
and our original variety $X=X_{10}\subset\PP^7$ is the projection of $X^{\sharp}$ from $P\in\PP^8$.
Let $\HHH_{Y,\Gamma}$ be a linear system of hyperplane sections of our cubic $Y$ passing through $\Gamma$.
Then $\Gamma=\Bs\HHH_{Y,\Gamma}$.
Let $\sigma:\tilde Y\to Y$ be the blowup of $\Gamma$ and let $\HHH_{\tilde Y,\Gamma}$ be the proper (birational) transform of $\HHH_{Y,\Gamma}$. Then $\tilde Y$ has only isolated singularities of type \type{cDV} (cf. \cite[Theorem~4]{Cutkosky-1988}) and the linear system $\HHH_{\tilde Y,\Gamma}$ is base point free and defines a fibration $\nu:\tilde Y\to\PP^1$ into cubic surfaces. The anticanonical linear system $|-K_{\tilde Y}|$ is the proper transform $\LLL_{\tilde Y,\Gamma}$ of the linear system $\LLL_{Y,\Gamma}$ and
$\LLL_{\tilde Y,\Gamma}\sim \HHH_{\tilde Y,\Gamma}+\sigma^*\HHH_{Y,\Gamma}$. Hence $\LLL_{\tilde Y,\Gamma}$ defines a map to $X^{\sharp}\subset\PP^{8}$.
It follows that the divisor $-K_{\tilde Y}$ is ample, $\tilde Y\dashrightarrow X^{\sharp}$ is an isomorphism, and $\uprho(X^{\sharp})=2$. Moreover, $X^{\sharp}\subset\PP^{8}$ is an anticanonically embedded Fano threefold of genus 7 with at worst terminal Gorenstein singularities \cite{Shin1989}.
The fibers of $\nu$ are mapped to cubic surfaces at $X\subset\PP^7$.
This means that the variety $X$ is trigonal and contradicts our assumptions.
Theorem~\ref{th:plane} is proved.
\end{proof}

\section {$\QQ$-factorial Fano threefolds}
In this section, we classify factorial nonrational Fano threefolds.
We do not assume that the Picard number equals 1.
This is due to the fact that in the construction \eqref{eq:constr} 
the output variety $\hat X^{(n)}$ can have a Picard number $>1$.

\begin{theorem}
\label{th:factorial}
Let $X$ be a Fano threefold with terminal factorial singularities.
Suppose that there are no nontrivial extremal contractions $X\to X'$ of types \type{B_1}-\type{B_4}.
If $X$ is not rational, then for $X$ one of the following possibilities holds:
\begin{enumerate}
\renewcommand\labelenumi{\rm(\arabic{section}.\arabic {subsection}.\alph {enumi})}
\renewcommand\theenumi{\rm(\arabic {section}.\arabic {subsection}.\alph {enumi})}
\item\label{fact:iota=2}\label{fact:cubic}
$\uprho(X)=1$, $\iota(X)=2$, $\dd(X)\le3$, and if $\dd(X)=3$, then $X$ is nonsingular;

\item\label{fact:iota=1}
$\uprho(X)=1$, $\iota(X)=1$, $\g(X)\le 6$;
\item\label{fact:X14}
$\uprho(X)=1$, $\iota(X)=1$, $X=X_{14}\subset\PP^9$ is a smooth threefold of the genus $8$ \cite{IP99};
\item\label{fact:P1P2}
$\uprho(X)=2$, $\g(X)=4$, $X$ is a double cover of $X\to\PP^1\times\PP^2$ with branch divisor of bidegree $(2, 3)$;

\item\label{fact:P2P2:22}
$\uprho(X)=2$, $\g(X)=7$, $X$ is a divisor of bidegree $(2,2)$ on $\PP^2\times\PP^2$;
\item\label{fact:P2P2:2}
$\uprho(X)=2$, $\g(X)=7$, $X$ is a double cover of $X\to V$ with branch divisor $B\in|-K_V|$, where $V$ is a smooth divisor of bidegree $(1,1)$ on $\PP^2\times\PP^2$;
\item\label{fact:P1P1P1}
$\uprho(X)=3$, $\g(X)=7$, $X$ is a double cover of $X\to\PP^1\times\PP^1\times\PP^1$ with branch divisor
tridegree $(2,2,2)$.
\end{enumerate}
All smooth threefolds of the types described above are not rational except for the threefolds with $\uprho(X)=1$, $\iota(X)=1$, $\g(X)=6$,
for which only the nonrationality of a \emph{general} member of the family is known.
\end{theorem}

All described threefolds are deformations of nonsingular Fano threefolds \cite{IP99}, \cite[(8.1)]{Mori-Mukai1983}.

\begin{sexample}
Let $X=X_{10}\subset\PP^7$ be a Fano threefold with $\iota=1$ and $\g(X)=6$ having a factorial singular point $P$ of type \type{cA_1^n}.
By \cite[Theorem~2]{Prokhorov-factorial-Fano-e} or \cite[Prop.~7.11]{Jahnke-Peternell-Radloff-II} there is a Sarkisov link
\[
\xymatrix@R=8pt {
&\tilde X\ar[dl]\ar@ {-->}[r] & X'\ar[dr]
\\
X &&&\PP^2
}
\]
where $\pi:X'\to\PP^2$ is a 
conic bundle with discriminant curve $\Delta$ of degree $6$.
Here $\tilde X\to X$ is the blowup of the maximal ideal of the point $P$.
The singular locus of $X'$ consists of one point of type \type{cA_1^{n-2}}.
If $P\in X$ is an ordinary double point or is of type \type{cA_1^3}, then the conic bundle $\pi$ is standard and
the variety $X$ is not rational \cite{Beauville1977}.
\end{sexample}

\begin{proof}[Proof of Theorem~\xref{th:factorial}]
First suppose that $\uprho(X)=1$.
The situation where $\iota(X)>1$ were considered in \cite{Prokhorov-GFano-1}.
The cases $\iota(X)=1$ for $\g(X)=7$ and $\g(X)\ge 9$ were considered in \cite[\S~4]{Prokhorov-planes}.
We prove that in the case $\g(X)=8$ the variety $X$ is smooth.
Suppose that $X$ has a singular point $P$.
If the singularity $P\in X$ is of type \type{cA_1}, then $X$ is rational
by \cite[theorem~2]{Prokhorov-factorial-Fano-e}.
So we may assume that $P\in X$ is of a type worse than \type{cA_1}.
It is known (see, e.g., \cite[\S~4]{Prokhorov-planes}) that
there is a Sarkisov link $X\gets\tilde X\dashrightarrow X'\to\PP^2$, where $\pi:X'\to\PP^2$ is a 
conic bundle with discriminant curve $\Delta$ of degree $5$.
Here $\tilde X\to X$ is the blowup of a line that does not intersect $\Sing (X)$.
Moreover, $X'$ has a singularity $P'$ locally isomorphic to $P\in X$
(in particular, the type of $P'\in X'$ is worse than \type{cA_1}). As in the proof
\cite[Lemma~1]{Prokhorov-factorial-Fano-e}, we obtain that the multiplicity of the
discriminant curve $\Delta\subset\PP^2$ at the point $o:=\pi (P')$ is at least 3.
Let $\LLL$ be a pencil of lines on $\PP^2$ passing through $o$.
Consider a standard model $\pi^\bullet:X^\bullet\to S^\bullet$ of the fibration
$\pi$ \cite{Sarkisov-1982-e}.
Let $\Delta^\bullet $ be the discriminant curve.
We may assume that there is a (birational) morphism $\varphi:S^\bullet\to\PP^2$
and the proper transform $\LLL^\bullet$ of the pencil $\LLL$ is base point free.
Since the set-theoretic intersection of the general element $L\in\LLL$
and $\Delta$ consists of at most three points, we have $\LLL^\bullet\cdot\Delta^\bullet\le 3$.
According to the Iskovskikh theorem (see, e.g., \cite{Iskovskikh-1987} or
\cite[Proposition 5.6]{P:rat-cb:e}) the variety $X^\bullet $ is rational, a contradiction.

Consider the case of $\uprho(X)\ge 2$.
Then there are at least two different extremal contractions of $f:X\to Y$ and $ f':X\to Y'$ on $X$.
By our assumption and Corollary~\ref{cor:rat} they have types \type{E_5}, \type{C_1} or \type{D_1}.
Consider these cases consecutively.

\subsection*{Set-up}
In our situation, let $H:=-K_X$.
Then $-K_X^3=2g-2$, where $g=\g(X)$ is the genus of $X$.
We introduce the following notation: $H_Y$ and $H_Y'$ are the ample generators of $\Pic(Y)$ and $\Pic(Y')$, respectively.
Put $M:=f^*H_Y$ and $M':=f'^*H_Y'$.

\subsection*{Type \type{E_5}}
Suppose that one of the extremal contractions, say $f$, is of type \type{E_5}.
Let $E$ be its exceptional divisor. Then $E\simeq\PP^2$.
If another contraction $f'$ has a two-dimensional fiber $F$, then $F\cap E\neq\varnothing$.
So $E$ and $F$ meet each other along a curve that must be contracted by both $f$ and $f'$.
This is impossible. Therefore, $f'$ is of type \type{C_1}. Since $E\simeq\PP^2$ dominates $Y'$, we have $Y'\simeq\PP^2$.
In particular, $\uprho(X)=2$.
Since $\Pic(X)=\ZZ\cdot H\oplus\ZZ\cdot M'$, we can write
$E\sim a H- bM'$.
We have the following relations (see, e.g., \cite[Lemma~4.1.6]{IP99}):
\begin{eqnarray*}
M'^3 &=& 0,\quad H\cdot M'^2=2,\quad H^2\cdot M'=12-d',
\\
E^3 &=& a^3 (2g-2) -3a^2b (12-d')+6ab^2=4,
\\
H\cdot E^2 &=& a^2 (2g-2) -2ab (12-d')+2b^2=-2,
\\
H^2\cdot E &=& a (2g-2) -b (12-d')=1.
\end{eqnarray*}
This system of Diophantine equations has no solutions.

\subsection*{Type \type{D_1}}
Suppose that one of the extremal contractions, say $f$, is of type \type{D_1}, i.e. $f$ is a del Pezzo fibration.
Let $F$ be its fiber and let $d:=K_{F}^2$.
Then $\uprho(X)=2$ and the second contraction of $ f'$ must be a conic bundle over $Y'=\PP^2$.
Let $\Delta'\subset Y'$ be the discriminant curve and let $d':=\deg\Delta'$.
We have the following relations (see, e.g., \cite[Lemma~4.1.6]{IP99}):
\begin{equation*}
M^3=M'^3=M^2\cdot H=0,\quad H\cdot M'^2=2,\quad H^2\cdot M=d,\quad H^2\cdot M'=12-d'.
\end{equation*}
Since $\Pic(X)=\ZZ\cdot H\oplus\ZZ\cdot M=\ZZ\cdot H\oplus\ZZ\cdot M'$, we can write
$M'\sim a H- M$. Then
\begin{eqnarray*}
{\textstyle\frac1a} M'^3 &=& a^2 (2g-2) -3ad=0,
\\
M'^2\cdot H &=& a^2 (2g-2) -2ad=2.
\end{eqnarray*}
We obtain a unique solution: $a=1$, $d=2$, $g=4$, $d'=8$.
The morphism
\[
h=f\times f':X\to V\subset\PP^1\times\PP^2
\]
is finite of degree $M\cdot M'^2=2$.
The branch divisor is computed using the Hurwitz formula.
We get the case~\ref{fact:P1P2}

\subsection*{Types \type{C_1}-\type{C_1}}
Consider the case where $\uprho(X)=2$ and both contractions of $f$ and $ f'$ are conic bundles.
Then $Y\simeq Y'\simeq\PP^2$. Let $\Delta\subset Y$ and $\Delta'\subset Y'$ be the corresponding discriminant curves, let $d:=\deg\Delta$,
and let $d':=\deg\Delta'$.
We have the following relations (see, e.g., \cite[Lemma~4.1.6]{IP99}):
\[
M^3=M'^3=0,\quad H\cdot M^2=H\cdot M'^2=2,\quad H^2\cdot M=12-d,\quad H^2\cdot M'=12-d'.
\]
Since $\Pic(X)=\ZZ\cdot H\oplus\ZZ\cdot M=\ZZ\cdot H\oplus\ZZ\cdot M'$, we can write
$M'\sim a H- M$. Then
\[
{\textstyle\frac1a} M'^3=a^2 (2g-2) -3a (12-d)+6=0
\]
and, similarly, $ a^2 (2g-2) -3a (12-d')+6=0$. Subtracting the first equality from the second one, we obtain $d=d'$.
We also have
\[
12-d=H^2\cdot M=a (2g-2) - (12-d).
\]
We obtain a unique solution: $a=1$, $d=6$, $g=7$.
As above, the morphism
\[
h=f\times f':X\to V\subset\PP^2\times\PP^2
\]
is finite onto its image $V\subset\PP^2\times\PP^2$.
If it is birational, then $V$ is divisor of bidegree $(2,2)$ (since $M^2\cdot M'=M'^2\cdot M=2$).
Since by the adjunction formula $K_X=h^*K_V$, the map $h$ is an isomorphism in codimension 1.
By the Serre criterion, the variety $V$ is normal and $h$ is an isomorphism.
We get the case~\ref{fact:P2P2:22}.

Suppose that the morphism $h$ is not birational. Then it is of degree 2 and $V$ is a divisor of bidegree $(1,1)$
(again because $M^2\cdot M'=M'^2\cdot M=2$).
In this case, the variety $V$ is normal and, according to the Hurwitz formula, the branch divisor belongs to the anticanonical linear system.
We get the case~\ref{fact:P2P2:2}.

\subsection*{Case $\uprho(X)>2$}
Finally, we assume that $\uprho(X)>2$ and all extremal contractions on $X$ are conic bundles.
By Lemma~\ref{lemma:rho=3:negative}, the base of any extremal contraction $f:X\to Y$ is isomorphic to $\PP^1\times\PP^1$.
This implies that $\uprho(X)=3$. Let $ l $ be one of the generators $Y=\PP^1\times\PP^1$ and let $F:=\pi^{-1}(l)$. Then
$F$ is a del Pezzo surface of degree $d:=K_X^2\cdot F$.
Let $\{F_1,\dots, F_n\}$ be the set of all such divisors (for various extremal contractions and generator $\PP^1\times\PP^1$).
For definiteness, we put $F_1=F$.
Since any effective divisor is nef on $X$, the cone of effective divisors $\EFF(X)\subset\Pic(X)_{\mathbb R}=\mathbb R^3$ is polyhedral and generated by the classes of divisors $F_i$.
Moreover, the linear system $|F_i|$ defines a del Pezzo fibration $h_i:X\to\PP^1$.
Each face of the cone $\EFF(X)$ is generated by two divisors $F_i$ and $F_{i+1}$ (with appropriate numbering).
If $F_i\cap F_j=\varnothing$, then $F_j$ is contained 
in the fibers of $h_i$, i.e. $F_i=F_j$.
If $F_i\cap F_j\neq\varnothing$, then $F_i\cap F_j$ is a conic on $F_i$ (and $F_j$) and the linear system $|F_i+F_j|$
defines a contraction $X\to\PP^1\times\PP^1$ with fiber $F_i\cap F_j$.
Since the image of the restriction map $\Pic(X)\to\Pic(F_i)$ is two-dimensional, up to linear equivalence there are exactly two conics on $F_i$ of the form $F_i\cap F_j$. Therefore, $n=3$, i.e. the cone $\EFF(X)$ has exactly three edges and they are generated by the divisors $F_1$, $F_2$, $F_3$.

Since the group $\Pic(X)$ is generated by the classes of divisors $H$, $F_1$, and $F_2$, we can write
\begin{equation*}
F_3=a_1H-b_1F_1-c_1F_2,\qquad a_i,\, b_i,\, c_i\in\ZZ,\quad a_i>0.
\end{equation*}
Similarly,
\begin{equation*}
F_1=a_2H-b_2F_2-c_2F_3,
\end{equation*}
\begin{equation*}
F_2=a_3H-b_3F_3-c_3F_1.
\end{equation*}
Then
\begin{equation*}
F_1=(a_2-c_2a_1) H+c_2b_1F_1+(c_2c_1-b_2) F_2,
\end{equation*}
\begin{equation*}
F_2=(a_3-b_3a_1) H+(b_3b_1-c_3) F_1+b_3c_1F_2,
\end{equation*}
\begin{equation*}
a_2-c_2a_1=0,\quad 1=c_2b_1,\quad c_2c_1-b_2=0,\quad a_3-b_3a_1=0,\quad b_3b_1-c_3=0,\quad 1=b_3c_1.
\end{equation*}
We obtain
\[
c_1=c_2=c_3=b_1=b_2=b_3=1,\qquad a_1=a_2=a_3.
\]
Thus,
\[
aH=F_1+F_2+F_3.
\]
Then
\[
K_{F_1}^2=H^2\cdot F_1=\textstyle\frac 1a H\cdot (F_1+F_2+F_3)\cdot F_1=\textstyle\frac 4a=K_{F_2}^2=K_{F_3}^2.
\]
In particular, $K_{F_i}^2=1$, $2$ or $4$. If $K_{F_i}^2=1$, then the base locus of the linear system $|-K_X|$ consists of three rational curves which are
sections of the fibrations $h_i$. This contradicts Theorem~\ref{thm:bs}.
Similarly, if $K_{F_i}^2=2$, then the anticanonical map is of degree $2$ on each surface of $F_i$
and maps this surface to a plane. But then the anticanonical image of the variety $X$ contains three distinct families of planes.
Clearly this is impossible. Therefore, $K_{F_i}^2=4$. Then the morphism $h_1\times h_2\times h_3:X\to\PP^1\times\PP^1\times\PP^1$
is a double cover.
We get the case~\ref{fact:P1P1P1}. This completes the proof of the classification part of Theorem~\ref{th:factorial}.

The nonrationality of smooth three-dimensional cubic hypersurfaces (case~\ref{fact:cubic}) is well-known \cite{Clemens-Griffiths}.
Any variety of type~\ref{fact:X14} is birationally equivalent to a cubic \cite{IP99}.
In cases~\ref{fact:iota=2} nonrationality follows from the results of \cite{Voisin1988} and \cite{Grinenko2004}.
Varieties
\ref{fact:iota=1} are not rational according to
\cite{Iskovskih-Manin-1971b}, \cite{Iskovskikh-Pukhlikov-1996}, and
\cite{Beauville1977}.

Non-singular threefolds of types~\ref{fact:P1P2}-\ref{fact:P1P1P1} have structures of standard conic bundles with ``large'' discriminant curve.
Then nonrationality follows from the results of \cite{Beauville1977}, \cite{Tjurin1979}, \cite{Shokurov1983}, see also \cite{Alzati-Bertolini-1992a}.
\end{proof}

\def\cprime{$'$}

\end{document}